\pdfoutput=1

\documentclass[11pt, oneside]{amsart}
\usepackage[a4paper, total={5in, 9in}]{geometry}  
\usepackage{graphicx}
\usepackage{epstopdf}
\usepackage{subcaption}
\usepackage{lipsum} 
\usepackage{amsrefs}
\usepackage{sidecap}
\usepackage{float}

\makeatletter
\renewcommand\paragraph{\@startsection{paragraph}{4}%
  \z@ \z@ {-\fontdimen2\font}%
  {\normalfont\bfseries}} 
\makeatother
\usepackage{tgtermes}  
\usepackage{color}
\usepackage{mathrsfs}
\usepackage{tikz-cd}
\usepackage{stmaryrd}

\usepackage{enumitem}
\newlist{myQuoteEnumerate}{enumerate}{2}
\setlist[myQuoteEnumerate,1]{label=(\alph*)}
\setlist[myQuoteEnumerate,2]{label=(\alph*)}

\usepackage{amsmath,amssymb,latexsym}
\usepackage[utf8]{inputenc}
\usepackage[english]{babel}
\usepackage{booktabs}
\newtheorem{thm}{Theorem}[section]
\newtheorem{thmm}{Theorem}[section]
\newtheorem{cor}{Corollary}[section]
\newtheorem{lem}{Lemma}[section]

\newtheorem{defi}{Definition}[section]

\newtheorem{remark}{Remark}[section]

\usepackage{siunitx}
\usepackage{xcolor}
\usepackage{booktabs,colortbl, array}
\usepackage{pgfplotstable}

\definecolor{rulecolor}{RGB}{0,71,171}
\definecolor{tableheadcolor}{gray}{0.92}

\numberwithin{equation}{section}

\newcommand\quotient[2]{
        \mathchoice
            {
                \text{\raise1ex\hbox{$#1$}\Big/\lower1ex\hbox{$#2$}}%
            }
            {
                #1\,/\,#2
            }
            {
                #1\,/\,#2
            }
            {
                #1\,/\,#2
            }
    }

\definecolor{aurometalsaurus}{rgb}{0.43, 0.5, 0.5}
\definecolor{darkjunglegreen}{rgb}{0.1, 0.14, 0.13}
\definecolor{coolblack}{rgb}{0.0, 0.18, 0.39}
\definecolor{cobalt}{rgb}{0.0, 0.28, 0.67}

\usepackage{hyperref}
\hypersetup{
     colorlinks   = true,
     citecolor    = blue
}

\title[Similarity of Multibrot sets]{\bf Similarity between the Multibrot set and the Julia set of correspondences at Misiurewicz points}

\author{Carlos Siqueira }

\date{\today}

\address{Department of Mathematics, Institute of Mathematics and Statistics, Federal University of Bahia, Salvador -- BA, Brazil.}

\curraddr{Department of Mathematics, Institute of Mathematics and Computer Science, University of S\~ao Paulo, S\~ao Carlos -- SP, Brazil.}

\email{carlos.siqueira@ufba.br}

\begin{document}

\hypersetup{linkcolor=cobalt}

\begin{abstract} 
We study the fine structure of the parameter space of the unicritical family of algebraic correspondences \( z^r + c \), where \( r > 1 \) is a rational exponent. Building on Tan Lei's  result regarding the similarity between the Mandelbrot set and Julia sets in the quadratic family, we prove that the Julia set of the correspondence is asymptotically self-similar about every Misiurewicz point.
 Assuming that the transversality condition holds at a Misiurewicz parameter \( a \in \mathbb{C} \), we prove that the associated Multibrot set (which coincides with the Mandelbrot set when \( r = 2 \)) is asymptotically similar to the Julia set about \( a \). We provide an algebraic proof of the transversality condition when the correspondence is represented by the semigroup  $\langle z^2 +c, -z^2+c \rangle. $ For general exponents, experimental evidence supports the transversality condition,  with infinitely many small copies of the Multibrot set accumulating at every Misiurewicz parameter.

 \end{abstract}

\maketitle

\keywords{MSC-class 2020:  37F05, 37F10 (Primary)   37F32 (Secondary).}

\section{Introduction}

The study of complex dynamics experienced a remarkable expansion during the 1980s and 1990s, marked by deep connections between complex analysis, geometry, and dynamical systems. Among the advances of this period was the discovery of self-similarity and scaling phenomena in the parameter space of the quadratic family $f_c(z) = z^2 + c$. Eckmann and Epstein~\cite{eckmann1985scaling} rigorously established the existence of miniature copies of the Mandelbrot set $M$ accumulating near Misiurewicz points, revealing an intricate scaling structure on the boundary of $M$. Building on their foundational theory of polynomial-like maps, 
Douady and Hubbard~\cite{douady1985dynamics} developed a general framework 
that explains why miniature copies of the Mandelbrot set should appear 
in many families beyond the quadratic one.
  Subsequently, Tan Lei~\cite{TanLei} proved a striking result: near any Misiurewicz parameter $c_0$, the Mandelbrot set is asymptotically similar to the Julia set $J_{c_0}$, thereby uncovering a profound correspondence between structures in the parameter space and those in the dynamical plane.

In this paper, we study the family of holomorphic correspondences defined by  
\[
z \mapsto \sqrt[q]{z^p} + c,
\]  
where \( r = \frac{p}{q} > 1 \) is a rational exponent (see  section \ref{lkhjakjhsoibasdfcsd}). This family constitutes a natural generalization of the  classical quadratic family.  

We introduce the notion of the \emph{filled Julia set} \( K_c \), consisting of all points in the complex plane that have at least one bounded forward orbit under the correspondence. We also define a generalized Mandelbrot set \( M_{p,q} \), consisting of all parameters \( c \in \mathbb{C} \) for which \( 0 \in K_c \).

Our main result shows that both \( M_{p,q} \) and \( K_c \) are \emph{asymptotic  similar} about a Misiurewicz point $c$ (see Definition \ref{ljhlasihkjhasdfacvqeaf}).

\subsection{Experimental overview.} 
Figure~\ref{fig:ssdfsdfasdedggweg} shows the Multibrot set of the family \[\mathbf{f}_c(z) = z^{5/2}+c.\] 
Its boundary contains infinitely many Misiurewicz points; for instance,
\[
   a = -1.027124 + 1.141048i
\]
is such a point, and the corresponding filled Julia set $K_a$ is also displayed in 
Figure~\ref{fig:ssdfsdfasdedggweg}, together with a magnification of $M_{5,2}$ by $10^5$. 
The figure reveals infinitely many miniature copies of the Multibrot set near $a$, and,
crucially, illustrates the asymptotic similarity between $K_a$ and $M_{5,2}$ in a 
neighborhood of $a$.

\begin{figure}[h]
  \centering
  \begin{subfigure}[b]{0.45\linewidth}
    \centering
    \includegraphics[width=\linewidth]{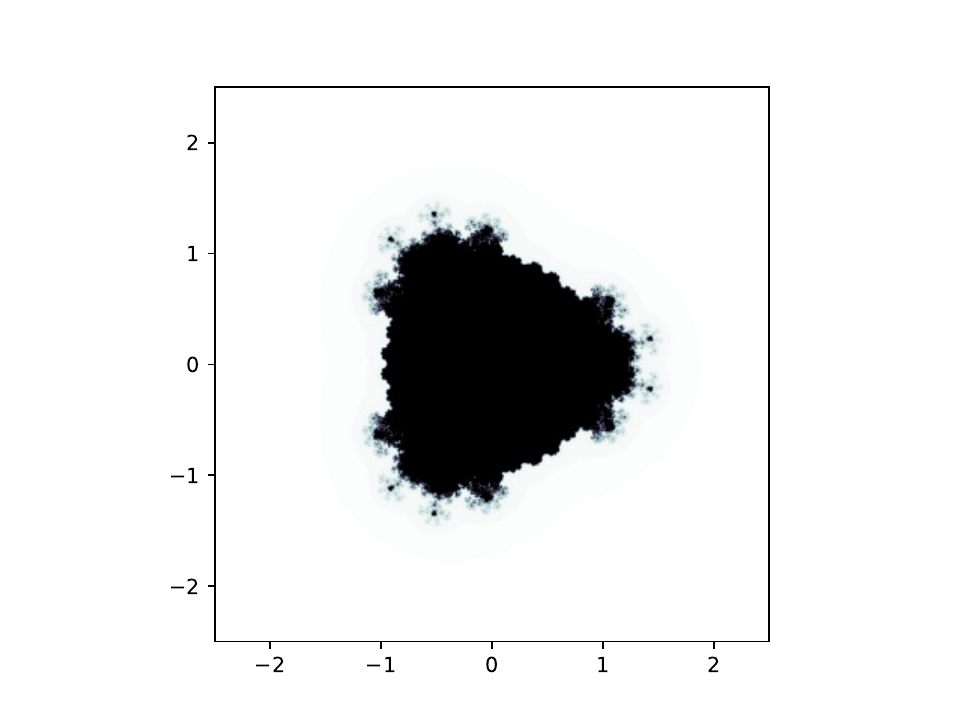}
    \caption{{\tiny $M_{5,2}$}} 
    \label{lkhasdbsfbsdfherhsg}
  \end{subfigure}%
  \hfill
   \begin{subfigure}[b]{0.55\linewidth}
    \centering
    \includegraphics[width=\linewidth]{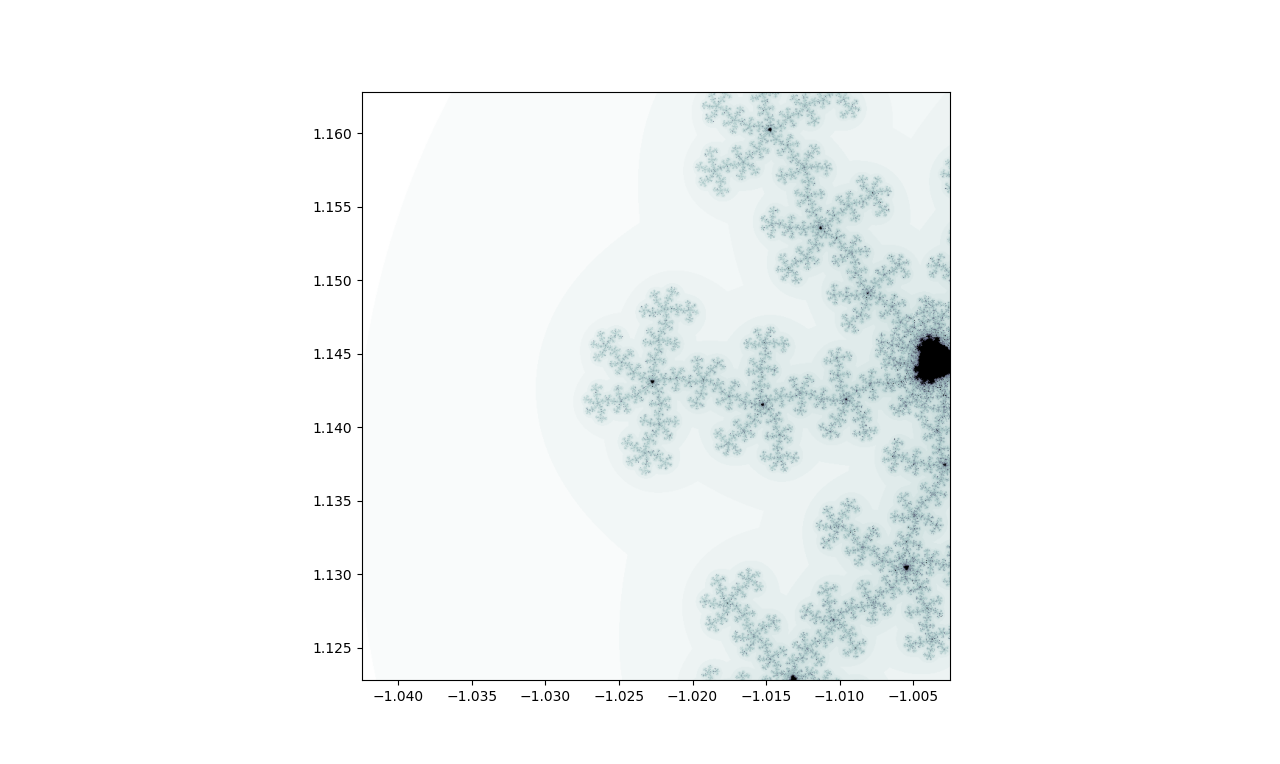}
    \caption{{\tiny $M_{5,2}$, magnified $10^3$ times at $a$
}}
    \label{lkhbsljhsoosdsdfgbnbwrfdgf}
  \end{subfigure}%
  \hfill
   \begin{subfigure}[b]{0.45\linewidth}
    \centering
    \includegraphics[width=\linewidth]{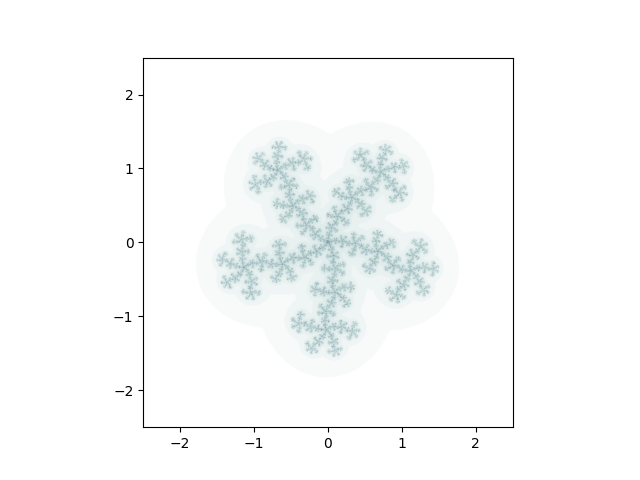}
    \caption{\tiny The filled Julia set $K_a$}
    \label{lkjkjapoasdfasdggfs}
  \end{subfigure}%
    \hfill
   \begin{subfigure}[b]{0.55\linewidth}
    \centering
    \includegraphics[width=\linewidth]{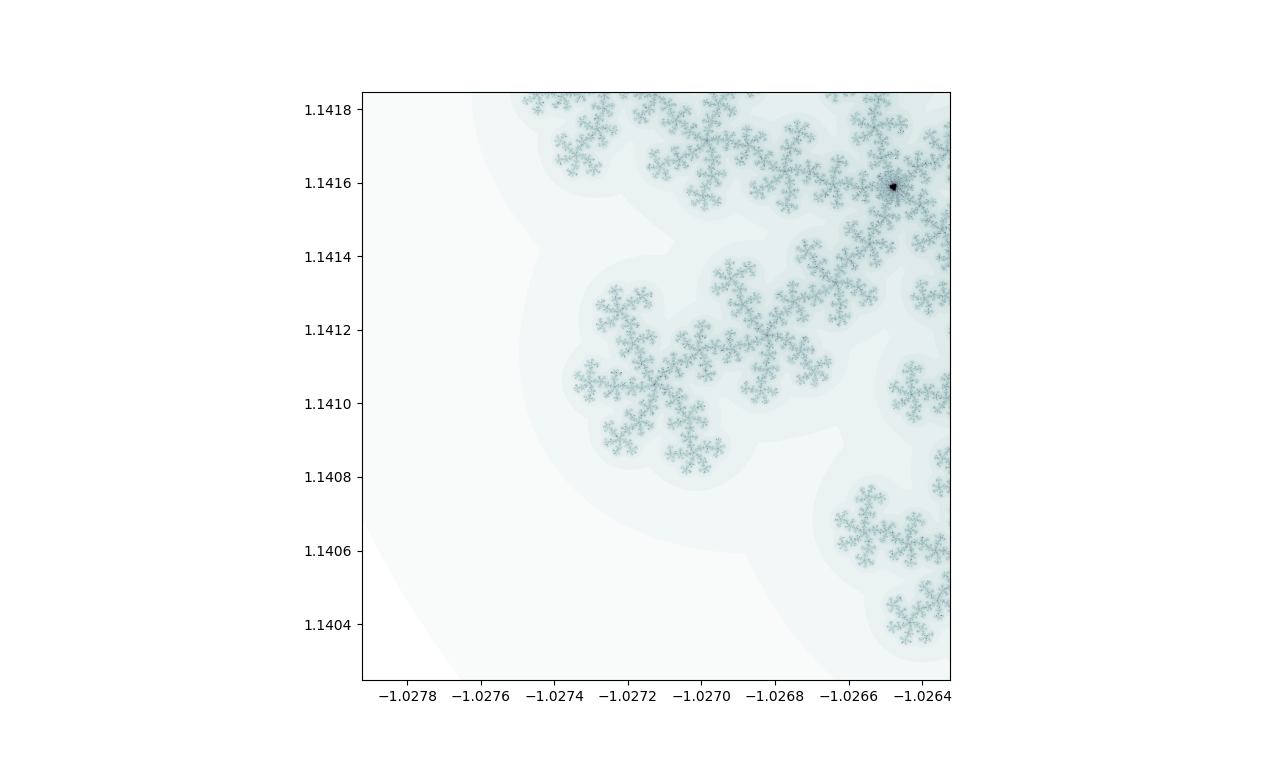}
    \caption{\tiny $M_{5,2}$, magnified $10^5$ times at $a$}
    \label{clhkkkclakssdggwbdfbwdg}
  \end{subfigure}%

  \caption{The Multibrot set $M_{5,2}$ and its magnifications at the Misiurewicz point \(a = -1.027124 + 1.141048i.\)  }
  \label{fig:ssdfsdfasdedggweg}
\end{figure}

This phenomenon extends to general exponents $p/q > 1$. 
Heuristically, detecting a sufficiently small copy (of size about $10^{-5}$ or less) 
already signals the presence of a Misiurewicz point. For example, $-2$ is a 
Misiurewicz point for the family $\sqrt{z^4}+c$, and
\[
   c_1 = -1.535 + 0.674i
\]
lies very close to another. In~\cite[Example~2.1]{siqueira2025suhyperbolic} a rigorous proof 
is given that infinitely many Misiurewicz points accumulate at $-2$ for this family.

The results presented in this paper are closely aligned with the conceptual framework introduced by Sullivan, commonly known as the \emph{Sullivan dictionary}. This framework establishes profound analogies between the iteration theory of rational maps and the theory of Kleinian groups. In particular, it draws parallels between objects such as Julia sets and limit sets, as well as between the Mandelbrot set and deformation spaces in Teichm\"uller theory. The dictionary further extends to holomorphic correspondences and has proved highly effective in interpreting many developments in holomorphic dynamics.

In the following sections, we present a  formal development of the results introduced so far through experimental illustrations.

\subsection{Algebraic correspondences.} 

Multi-valued maps  $z \mapsto w$ in the complex plane  defined implicitly by a polynomial equation $P(z, w) = 0$ in two complex variables are known as \emph{algebraic correspondences.} These systems generalize both rational maps and Kleinian groups, while introducing new layers of complexity. 
 Fatou and Julia, in the early decades of the twentieth century, recognized their significance in the development of a unified theory of one-dimensional complex dynamics.
However, the systematic study of algebraic correspondences began in the 1980s, with foundational work by Bullett and Penrose in the following decades~\cite{BullettNonlinearity, Bullett92, BPExp, Bullett1994}. In~\cite{Bullett1994}, Bullett and Penrose proposed the celebrated conjecture that the connectedness locus of the family of algebraic correspondences defined by
\[
\left( \frac{a w - 1}{w - 1} \right)^{2}
+ \left( \frac{a w - 1}{w - 1} \right)
  \left( \frac{a z + 1}{z + 1} \right)
+ \left( \frac{a z + 1}{z + 1} \right)^{2}
= 3
\]
is homeomorphic to the Mandelbrot set.

Recently, there has been significant progress in the field, with important contributions coming from several directions. Notably, Bullett and Lomonaco proved the long-standing conjecture originally proposed in~\cite{Bullett1994}, as shown in~\cite{BL19, BULLETT2024109956}. In addition, Lee, Lyubich, Makarov, and Mukherjee~\cite{Mukherjee} made a major contribution by studying a family $\mathcal{S}$ of Schwarz reflections,   showing that the abstract connectedness locus of $\mathcal{S}$ is homeomorphic to the abstract parabolic Tricorn, the anti-holomorphic Mandelbrot set. See also~\cite{bullett2024mating, luo2025teichm, lyubich2024antiholomorphic, lee2023dynamics} for more recent developments in the area.

\subsection{The unicritical family}  \label{lkhjakjhsoibasdfcsd}
The multivalued function \( w = \sqrt[n]{z} \), the \( n \)th root of \( z \), is simply the algebraic correspondence defined by the equation \( w^n - z = 0 \).
Consider the correspondence \( \mathbf{f}_c: \mathbb{C} \to \mathbb{C} \) given by
\begin{equation}\label{lkjdllkjhalsoed}
\mathbf{f}_c(z) = \sqrt[q]{z^p} + c,
\end{equation}
where \( p > q>0. \) In other words, \( \mathbf{f}_c \) is a \( (p\!:\!q) \) correspondence $(w-c)^q=z^p$; except at the critical value \( c \) and the critical point \( 0 \), each point has exactly \( q \) images and \( p \) pre-images.

Forward and backward orbits are defined in the natural way. Due to the multivalued nature of the correspondence, every point admits uncountably many forward orbits.

Previous works on the family \eqref{lkjdllkjhalsoed} relate to holomorphic motions \cite{SS17}, geometric rigidity of the post-critical set \cite{ETDS22}, and estimates on the Hausdorff dimension of the Julia set using an analogue of the Bowen formula for such correspondences \cite{Proc22}.

\subsection{Asymptotic similarity in the dynamical plane.} Any sequence $(z_i)_{i=0}^{\infty}$ satisfying $z_{i+1}\in \mathbf{f}_c(z_i)$ for every $i$ is called an orbit of the correspondence $\mathbf{f}_c$. An orbit is said to be bounded if it is contained in a compact subset of the complex plane.

The \emph{filled Julia set} $K_c$  of the correspondence  \eqref{lkjdllkjhalsoed} consists of every $z$ having at least one bounded forward orbit.

In the following definition, $M_{2, 1}$ corresponds to the Mandelbrot set.

\begin{defi}[\bf Multibrot set] $M_{p, q}$ is the set of parameters $c$ for which $0\in K_c.$

\end{defi}

For the  quadratic family \( f_c(z) = z^2 + c \), a parameter \( c \) is called a \emph{Misiurewicz point} if the critical point \( 0 \) is strictly pre-periodic. In the setting of the algebraic correspondence \( \mathbf{f}_c \), we adopt the following  definition.

\begin{defi}[\bf Misiurewicz point]\label{ljhlasihkjhasdfacvqeaf}\normalfont
A parameter \( a\in \mathbb{C} \) is called a \emph{Misiurewicz point} if the critical point \( 0 \) has a unique bounded forward orbit under the correspondence \( \mathbf{f}_a \), and this orbit is strictly pre-periodic.
\end{defi}

According to the definition of a Misiurewicz point, the point \( z_0 = a \) has a unique bounded forward orbit, which is necessarily pre-periodic to a  cycle $\alpha_a$:
\[
a = z_0 \xrightarrow{\mathbf{f}_a} z_1 \xrightarrow{\mathbf{f}_a} \cdots \xrightarrow{\mathbf{f}_a} z_{\ell} \xrightarrow{\mathbf{f}_a} \cdots \xrightarrow{\mathbf{f}_a} z_{\ell + n} = z_{\ell}.
\]
According to \cite[Theorem D]{siqueira2025suhyperbolic}, the cycle $\alpha_a$, which begins at \( z_{\ell} \) and ends at \( z_{\ell + n} \), is repelling in the sense that the composition of univalent branches of \( \mathbf{f}_a \) along the cycle defines a holomorphic map with a repelling fixed point at \( z_{\ell} \). 
We will  construct a \emph{holomorphic motion} of this repelling cycle. As a result, for each \( c \) sufficiently close to \( a \), there exists a sequence
\begin{equation}\label{asdfasdflhjlkkjhxc}
\xi(c) \xrightarrow{\mathbf{f}_c} z_1(c) \xrightarrow{\mathbf{f}_c} \cdots \xrightarrow{\mathbf{f}_c} z_{\ell}(c) \xrightarrow{\mathbf{f}_c} \cdots \xrightarrow{\mathbf{f}_c} z_{\ell + n}(c) = z_{\ell}(c),
\end{equation}
where each \( z_j(c) \) and \( \xi(c) \) is a holomorphic function of \( c \), with \( \xi(a) = a \) and $z_j(a)=z_j$ for every $j.$

One of the key (and technically subtle) steps in the proof of self-similarity is to show that the sequence \eqref{asdfasdflhjlkkjhxc} is the \emph{only} bounded forward orbit of \( \xi(c) \) under \( \mathbf{f}_c \) for all \( c \) sufficiently close to \( a \). In other words, the defining condition for Misiurewicz points is stable under small perturbations of the parameter: although \( \xi(c) \neq c \), the point \( \xi(c) \) satisfies the same uniqueness condition as \( a = \xi(a) \), namely, that it has a unique bounded forward orbit, which lands on a repelling cycle.
Indeed, for each such parameter \( c \), the cycle in \eqref{asdfasdflhjlkkjhxc}  beginning at \( z_{\ell}(c) \) remains repelling, with multiplier \( \lambda(c) \).

The concept of asymptotic similarity, along with related definitions, is presented in Section~\ref{adsflkajlksjdfaoibadadf}.
The following result is stated in full generality as Theorem~\ref{gdasoiu}.

\begin{thmm}[\bf Asymptotic self-similarity of \( K_c \)]
Let \( a \in \mathbb{C} \) be a Misiurewicz point for the family defined by \eqref{lkjdllkjhalsoed}. For every parameter \( c \) in a neighborhood of \( a \),  \( K_c \) is asymptotically \( \lambda(c) \)-self-similar about each point  in the orbit  \eqref{asdfasdflhjlkkjhxc}. \end{thmm}

\noindent In particular, \( K_a \) is asymptotic self-similarity about the Misiurewicz point \( a \in K_a \), as well as about every point in the associated repelling cycle $\alpha_a.$

\subsection{The transversality condition.}

According to \eqref{alsokkkdfoiwegasc}, 
if \( a \) is a Misiurewicz point for the family \( \mathbf{f}_c \) given by \eqref{lkjdllkjhalsoed}, then for every \( c \) sufficiently close to \( a \), there exists a univalent map \( g_c \) defined on a neighborhood of \( \xi(c) \), given by the composition of branches of $\mathbf{f}_c,$ which maps \( \xi(c) \) to the \( z_{\ell}(c) \) in  \eqref{asdfasdflhjlkkjhxc}.  See \eqref{alsokkkdfoiwegasc} for more details. The domain of \( g_c \) contains \( c \), and \( g_a(a) = z_{\ell} \).

Let \( h_c \) denote the composition of univalent branches of \( \mathbf{f}_c \) along the repelling cycle starting at \( z_{\ell}(c) \) and ending at \( z_{\ell + n}(c) \), as described in  \eqref{asdfasdflhjlkkjhxc} and \eqref{afsdfoijlkjblsdf}. Then \( h_c \) is a holomorphic map with a repelling fixed point at \( z_{\ell}(c) \). It can also be shown that \( g_c(c) \) lies in the domain of \( h_c \) for every \( c \) sufficiently close to \( a \). It follows that the function
\[
w(c) = h_c(g_c(c)) - g_c(c)
\]
is holomorphic on a neighborhood of \( a \), and satisfies \( w(a) = 0 \).

\vspace{0.2cm}
\noindent\textbf{Transversality condition (Definition~\ref{asdljhafspdoifasdf}).} \textit{We say that the family \( \mathbf{f}_c \) satisfies the \emph{transversality condition at \( a \)} if \( w'(a) \neq 0 \).}

\vspace{0.2cm}
\noindent\textbf{Conjecture.} \textit{The transversality condition holds at every Misiurewicz point for all integer exponents \( (p, q) \) with \( p > q >1  \) in the family}
\[
\mathbf{f}_c(z) = \sqrt[q]{z^p} + c.
\]

The transversality condition for polynomials  was first introduced by Douady and Hubbard \cite{DH84, douady1985dynamics}.

Currently, this conjecture has been proven only in the case \( (p, q) = (4, 2) \); see Theorem~\ref{dfasdfefd}. The challenge lies in the fact that transversality can, in principle, be approached through analytic, algebraic, or dynamical means. Generalizations of the dynamical proof from Douady and Hubbard \cite{DH84, douady1985dynamics}  are not applicable to   $\mathbf{f}_c.$ For now, our proof in the case $(p, q)=(4, 2)$ relies on algebraic methods; see section \ref{jgc}.

\subsection{Asymptotic similarity between $M_{p, q}$ and $K_a$.}We are now in a position to state the main result establishing the asymptotic similarity between the Multibrot set \( M_{p, q} \) and the filled Julia set \( K_a \) at a Misiurewicz parameter \( a \). The theorem is formulated in general terms and applies whenever the transversality condition is satisfied. In this sense, \emph{the problem of proving similarity has been reduced to verifying transversality.}

Recall from \eqref{asdfasdflhjlkkjhxc} that a Misiurewicz point \( a \) is mapped, after finitely many steps, to a repelling cycle $\alpha_a$ beginning at a point \( z_{\ell} \) with multiplier \( \lambda_a \).  The next theorem will be restated in full as Theorem~\ref{gaeiogc}.

\begin{thmm}[\bf Similarity between Multibrot and Julia Sets]\label{adlkjaladslfkjaoocvsf}
Let \( a \in \mathbb{C} \) be a Misiurewicz parameter for the family \eqref{lkjdllkjhalsoed}, and suppose that the transversality condition holds at \( a \). Then the Multibrot set \( M_{p,q} \) and the filled Julia set \( K_a \) are asymptotically self-similar about the point \( a \), with common scaling factor \( \lambda_a \). Moreover, the corresponding limit models coincide up to multiplication by a nonzero complex constant.
\end{thmm}

As a direct consequence of the transversality condition, already established for the case \( (p, q) = (4, 2) \), we obtain the following result. It will later be restated in full as Corollary~\ref{adflhlkjhasdlkjhfqewfdasdfc}.

\begin{thmm}[\bf Similarity for quadratic correspondences]
Let \( a \) be a Misiurewicz point for the family \( \mathbf{f}_c \) given by \eqref{lkjdllkjhalsoed}, with \( (p, q) = (4, 2) \). Then  \( K_a \) and  \( M_{4,2} \) are asymptotically self-similar about the point \( a \), with the same scaling factor, and the same limit model up to multiplication by a nonzero complex constant.\end{thmm}
\vspace{0.2cm}

\paragraph{\bf Remark.} By Theorem D in \cite{siqueira2025suhyperbolic}, if $a$ is Misiurewicz then \( K_a = J_a \); thus, repelling cycles are dense in $K_a$, which is a crucial ingredient in the proof of Theorem \ref{adlkjaladslfkjaoocvsf}.  It is worth noting that in the classical quadratic case, Tan Lei's proof of self-similarity relies on certain results of Douady and Hubbard, including the transversality property and the density of repelling periodic points in the filled Julia set. For algebraic correspondences, however, such results are not immediate: the Julia and Fatou sets are no longer completely invariant, and many classical tools fail to apply directly. This adds substantial complexity to the adaptation of Tan Lei's approach in our setting.

\section{Preliminaries}

The correspondence \eqref{lkjdllkjhalsoed} comes with an exponent $p/q > 1$, determined by its expression $\mathbf{f}_c(z) = \sqrt[q]{z^p} + c$. For any $\lambda > 1$, the real equation
\[
x^{p/q} - \lambda x - |c| = 0
\]
has two solutions, denoted by $x_1(\lambda) < x_0(\lambda)$. We say that $R_c > 0$ is an \emph{escaping radius} for $\mathbf{f}_c$ if $R_c > x_0(\lambda)$ for some $\lambda > 1$.  The properties of the escaping radius are described in~\cite[Section~2]{ETDS22}. We summarize a few of them as follows.

The \emph{basin of infinity}, denoted by $B_c(\infty)$, is the set of all points $z$ such that all forward orbits
\begin{equation}\label{asdkljhalkjhasdfasdfabfagsfhjwtwhdf}
z \xrightarrow{\mathbf{f}_c} z_1 \xrightarrow{\mathbf{f}_c} z_2 \xrightarrow{\mathbf{f}_c} \cdots
\end{equation}
converge to $\infty$ in the spherical metric. It is clear that $B_c(\infty)$ is forward invariant: if $z$ belongs to $B_c(\infty)$, then so does every forward image of $z$ under $\mathbf{f}_c$. Let $\mathcal{B}_R = \{z \in \mathbb{C} : |z| < R\}$. The complement of the basin of infinity is, by definition, \emph{the filled Julia set $K_c$}. 
According to~\cite[Theorem~2.1]{ETDS22}, the complement of $\mathcal{B}_R$ is also forward invariant and
\begin{equation}\label{safadfwdgasdvcqeqwsdfs}
\hat{\mathbb{C}} \setminus \mathcal{B}_R \subset B_c(\infty) = \hat{\mathbb{C}} \setminus K_c.
\end{equation} Hence $z\in K_c$ precisely when $z$ has at least one bounded forward orbit.

It can be shown that the escaping radius $R_c > 0$ is stable under small perturbations of the parameter. That is, if $R_a$ is an escaping radius for $\mathbf{f}_a$, then $R_c = R_a$ remains an escaping radius for $\mathbf{f}_c$, for every $c$ sufficiently close to $a$.   The iterates $\mathbf{f}_c^n$ are defined in the natural way: we say that $w \in \mathbf{f}_c^n(z)$ whenever there exists a finite forward orbit as in~\eqref{asdkljhalkjhasdfasdfabfagsfhjwtwhdf} with $z_n = w$.  The inverse $\mathbf{f}_c^{-1}$ is understood as the multivalued map $w \mapsto z$, sending $w$ to $z$ if and only if $\mathbf{f}_c$ maps $z$ to $w$. 

Note that $\mathbf{f}_c^{-1} \circ \mathbf{f}_c$ is not the identity: the correspondence $\mathbf{f}_c$ sends a nonzero point $z$ to $q$ values in the plane, and each of these has $p$ pre-images under the inverse correspondence. Therefore, in general, the composition $\mathbf{f}_c^{-1} \circ \mathbf{f}_c(z)$ consists of $p \cdot q$ points. We can also define $\mathbf{f}_c^n(A)$ and $\mathbf{f}_c^{-n}(A)$ for any subset $A$ of the complex plane. The image $\mathbf{f}_c^n(A)$ is the union of all sets $\mathbf{f}_c^n(z)$ with $z \in A$, while the preimage $\mathbf{f}_c^{-n}(A)$ is the union of all sets $\mathbf{f}_c^{-n}(z)$ with $z \in A$.

As shown in~\cite[equation~(3)]{ETDS22}, if $R$ is an escaping radius for $\mathbf{f}_c$, then
\begin{equation}\label{alhasdfoilkhlkjhasdfc}
K_c = \bigcap_{n > 0} \mathbf{f}_c^{-n}(\mathcal{B}_R).
\end{equation}

The following result is a well-known consequence of  \eqref{alhasdfoilkhlkjhasdfc} and \cite[Lemma 2.2]{ETDS22}.

\begin{thm}\label{adsfasdgglkppsdfsdf}
Let $R$ be an escaping radius of $\mathbf{f}_c.$ If $z \notin K_c$, then there exists an integer $k \geq 1$ such that $\mathbf{f}_c^k(z) \subset \mathbb{C}{\setminus}\mathcal{B}_R$.
\end{thm}

\subsection{Repelling cycles.}  Any univalent map $f$ satisfying $f(z) \in \mathbf{f}_c(z)$ for every $z$ in its domain is called a \emph{univalent branch of $\mathbf{f}_c$}.
 A forward orbit $(z_i)_{i=0}^{\infty}$ of the correspondence $\mathbf{f}_c$ is said to form a cycle of period $n$ if $z_i = z_{i+n}$ for all $i$. Except when $z_i$ is the critical point (so that $z_{i+1} = c$) there exists a unique univalent branch of $\mathbf{f}_a$ mapping $z_i$ to $z_{i+1}$. By composing these local branches along the cycle, we obtain a univalent map $f$ with a fixed point at $z_0$.

In this setting, the usual classification applies: the cycle is called \emph{repelling, attracting, super-attracting, geometrically attracting, indifferent}, and so on, according to the behavior of $f$ at the fixed point $z_0$. Known linearization results for holomorphic functions near fixed points extend naturally to this setting through this identification.

If a cycle contains the critical point, we call it a \emph{critical cycle}. In such cases, the usual notion of multiplier does not apply. Nevertheless, we adopt the convention of referring to all critical cycles as \emph{super-attracting}, even though a multiplier is not defined. It is straightforward to verify that any cycle not passing through the critical point can never be super-attracting. If $z$ belongs to a repelling cycle, then $z$ is \emph{repelling periodic point.}

\begin{defi}[\bf Julia set]\label{adfadweerdccgqwe}\normalfont
The \emph{Julia set} $J_c$ is defined as the closure of the union of all repelling cycles of $\mathbf{f}_c$.
\end{defi}

As a consequence of \cite[Lemma~2.1]{SS17} and equation~\eqref{alhasdfoilkhlkjhasdfc}, the set $K_c$ is compact. Since $K_c$ contains every cycle, it follows that $J_c \subset K_c$, and thus $J_c$ is also compact.

\subsection{Similarity.} \label{adsflkajlksjdfaoibadadf} Assume that $A$ and $B$ are nonempty closed subsets of  $\mathbb{C}.$ If $c$ belongs to $A$, then $T_{-c}A$ is the translate of $A$ by $-c$, hence  $0\in T_{-c}A.$ As usual, let $\mathbb{D}_s$ denote the set of all complex numbers with $|z| <s.$

 For every $r>0$, the set
\[A_r=  (A\cap \overline{\mathbb{D}}_r) \cup \partial \mathbb{D}_r \] is compact. Thus we are allowed to calculate the \emph{Hausdorff distance} between $A_r$ and $B_r$, denoted by $d_H(A_r, B_r).$ Now let $\lambda$ be any nonzero complex number with $|\lambda| >1.$ We think of $\lambda$ as an expanding factor $|\lambda|,$ followed by a rotation by the angle $\arg (\lambda).$ We say that $A$ is \emph{self-similar about} $c\in A$ \emph{with scale} $\lambda$ (or equivalently, $\lambda$-\emph{self-similar}) if \begin{equation} \label{lsopijlkasopijdf}(\lambda T_{-c} A)_r = (T_{-c}A)_r, \end{equation} for some $r>0.$ In this case, \eqref{lsopijlkasopijdf}  is also  true if we replace $r$ by any value in $(0, r).$  

The set $A$ is said to be \emph{asymptotically $\lambda$-self-similar about $c\in A$ with limit model $B$} if the sequence of sets $(\lambda^n T_{-c} A)_r$ converges to $B_r$ in the Hausdorff topology of compact sets, for some $r>0.$ In this case, the same property remains valid for any constant in   $(0,r),$ replacing $r$ accordingly.  
It is also common to say that \emph{$A$ is asymptotically self-similar about $c$ with scale $\lambda$ and limit model $B.$ } It can be shown that the limit model $B$ contains $0$ and is $\lambda$-self-similar about $0;$ moreover, each $\lambda^nB$ is also a limit set.

Suppose that $c$ belongs to $A\cap B.$ The sets $A$ and $B$ are  \emph{asymptotically similar about $c$} provided \begin{equation} \label{asdopslkjdsdflkkc}d_{H}( (\lambda T_{-c} A)_r, (\lambda T_{-c} B)_r     )  \to 0\end{equation} as $|\lambda| \to \infty$, with $\lambda \in \mathbb{C}.$ It is clear that  if \eqref{asdopslkjdsdflkkc} holds for some $r>0$, then it also holds for any $s$ in $(0, r),$ replacing $r$ with $s$).

The following result is a consequence of Proposition 2.4 of \cite[p.593]{TanLei}.

\begin{thm}\label{gjdqwefd} Let $A_1$, $A_2$ and $B$ be nonempty closed subsets of $\mathbb{C}.$ Suppose that $A_1$ is asymptotically $\lambda$-self-similar about $c\in A_1$, with limit model $B.$ Assume $f$ is a univalent map sending a relatively open subset of $A_1$ containing $c$ onto a relatively open subset of $A_2$ containing $f(c).$  Then $A_2$ is asymptotically self-similar about $f(c)$ with  the same scale $\lambda$ and limit model $f'(c) B.$

\end{thm}

\section{Transversality}

This section is devoted to the study of the transversality condition for the family  $\mathbf{f}_c$ given by \eqref{lkjdllkjhalsoed}.

Let $b\in \mathbb{C}.$ If $f$  is a  univalent branch of the correspondence $\mathbf{f}_b, $ with $f$ locally defined at a nonzero point $z_0$, then $f=\phi +b,$ where $\phi$ is a univalent branch of $\mathbf{f}_0.$ Then $f_c=\phi +c$ is \emph{the perturbed branch}  in the sense that, for every $c$ sufficiently close to $b$, the holomorphic map $f_c$ is a branch of $\mathbf{f}_c,$ with $f_{b}=f$ corresponding to the base point of the holomorphic family $c\mapsto f_c$, that is,  $(c, z) \mapsto f_c(z)$ is holomorphic on a neighborhood of $(b, z_0)$ in $\mathbb{C}^{2}.$ 
 If $f$ comes with an index such as $f_i$ then we  denote the perturbed branch by $f_{i,c}$.  For a composition of branches $f=f_n \circ \cdots \circ f_1$, the corresponding perturbed branch is defined by
  \begin{equation}\label{asdfasoiwellkjsdf}
f_c = f_{n, c} \circ \cdots \circ f_{1, c}. 
\end{equation}
 If $n>1,$  $f_c$ depends on  the choice of each $f_i$ in the composition. However, if the maps $f_i$ are implicit from the context, then the perturbed branch $f_c$ is uniquely determined. If $n=1$, the situation is much simpler and we have only one possible perturbation $f_c$ associated to each $f.$

\begin{lem}[\bf Holomorphic motion of a repelling cycle] \label{ghdcsdffewdf} Let $b \in \mathbb{C}.$  Let $(z_i)_0^n$ be a repelling cycle of period $n$ of $\mathbf{f}_b$ with univalent branches $f_i$  of \( \mathbf{f}_b\) sending $z_{i-1}$ to $z_i$, for every $i.$ Extend the cycle and the sequence of branches to  $n$-periodic infinite sequences $(z_i)_0^\infty$ and $(f_i)_1^\infty,$ respectively. Then 
\begin{equation}\label{slkjspoijdfwdllm}h_j=f_{n+j} \circ \cdots \circ f_{j+1}
 \end{equation} 
  has a repelling fixed point at $z_j$, for every nonnegative integer $j.$ Let $h_{j,c}$ be the associated perturbed branch of $\mathbf{f}_c^n,$ as described in \eqref{asdfasoiwellkjsdf}. For every $j$ in $\mathbb{Z} \cap [0, n]$:

\begin{enumerate}[label=(\roman*)]

\item $(c, z) \mapsto h_{j, c}(z)$ is holomorphic on a neighborhood of $(b, z_j)$;

\item there exists a unique holomorphic function $\tilde{z}_j$ defined on a neighborhood of $b$ such that $\tilde{z}_j(b)=z_j$ and  $h_{j,c}$ has a   fixed point at $\tilde{z}_j(c),$ for every $c$ in the domain of $\tilde{z}_j;$ and

\item the perturbed branch $f_{j,c}$ sends $\tilde{z}_{j-1}(c)$ to $\tilde{z}_j(c)$ and
\[ \tilde{z}_0(c) \xrightarrow{f_{1,c}} \tilde{z}_1(c) \mapsto \cdots \xrightarrow{f_{n,c}} \tilde{z}_n(c)=\tilde{z}_0(c)\]
\noindent is a repelling cycle of the holomorphic map $z\mapsto h_{j,c}(z),$ for every $c$ in a neighborhood of $b.$

\end{enumerate}

\end{lem}

\begin{proof}  To the periodic sequence of maps \( (f_j)_{1}^{\infty} \), we associate an \( n \)-periodic sequence of perturbed branches \( (f_{j,c})_{1}^{\infty} \).
 Then   \begin{equation}\label{asdpoioasdfpowikdf}h_{j,c} = f_{n+j,c} \circ \cdots \circ f_{j+1,c}. \end{equation} 
 Let $F_j(c, z) = h_{j,c}(z) -z.$    The partial derivative $\partial F_j / \partial z$ at $(b,z_j)$  is $\lambda_j -1 \neq 0$, where $\lambda_j$ is the multiplier of the repelling fixed point  $z_j$ of $h_{j}.$  By the Implicit Function Theorem, there exists a unique map $\tilde{z}_j$ defined on a neighborhood of $b$ such that $\tilde{z}_j(b) =z_j$ and $\tilde{z}_j(c)$ is a fixed point of $h_{j,c}$, for every $c$ in the domain of $\tilde{z}_j.$ Since $(c, z) \mapsto h_{j,c}(z)$ is holomorphic, we may assume, without loss of generality, that $\tilde{z}_j(c)$ is a repelling fixed point of $h_{j,c},$ for every $c$ in the domain of $\tilde{z}_j.$ 
 
 It remains to show that $f_{j,c}$ sends $\tilde{z}_{j-1}(c)$ to $\tilde{z}_j(c).$ Indeed, let  $\tilde{w}_j(c)$ denote $f_{j,c}(\tilde{z}_{j-1}(c)).$ We will show that $\tilde{w}_j$ and $\tilde{z}_j$ coincide on   a neighborhood of $b$, using the uniqueness part of the Implicit Function Theorem. Notice that both functions have the same value at $b.$  By \eqref{asdpoioasdfpowikdf} and the periodicity of $(f_{j,c})_j,$ $$h_{j,c}(\tilde{w}_j(c)) = f_{n+1,c}\circ h_{j-1,c} (\tilde{z}_{j-1}(c)) =f_{n+j,c}(\tilde{z}_{j-1}(c)) = \tilde{w}_j(c).$$
 
 The proof is complete. 
 \end{proof}

\noindent {\bf Stability of the repelling cycle of a Misiurewicz point.} For the following lemma, we will assume that $a\in \mathbb{C}$ is a Misiurewicz point for the family \eqref{lkjdllkjhalsoed}, with an associated pre-periodic orbit
\begin{equation}\label{alkjsokjdkkdf}a \xrightarrow{f_1} z_1  \xrightarrow{f_1} z_2 \cdots \xrightarrow{f_{\ell}} z_{\ell} \to \cdots \xrightarrow{f_{\ell +n}} z_{\ell +n} = z_{\ell} \to \cdots \end{equation}
where $(f_j)_1^\infty$ is a pre-periodic  sequence of univalent branches and $z_\ell$ is the first point  that belongs to the repelling cycle $(z_i)_\ell^{\ell +n}$ of period $n$, which we denote by $\alpha_a$ or $\alpha(a).$ (In the next paragraph, $\alpha$ will be a function assigning  to each $c$ a repelling cycle $\alpha(c)$ with $\alpha(a)=\alpha_a$). By Lemma \ref{ghdcsdffewdf},  the following cycle depends holomorphically on $c$:
\begin{equation}\label{oaoijsmdfsebdf}\tilde{z}_\ell(c) \xrightarrow{f_{\ell +1,c}} \tilde{z}_{\ell +1}(c) \to \cdots \xrightarrow{f_{\ell+n,c}}  \tilde{z}_{\ell +n}(c) = \tilde{z}_{\ell}(c) \xrightarrow{f_{\ell +n+1,c}} \cdots\end{equation}
where $f_{j,c}$ is the perturbed branch of $f_j.$ Notice that  $(f_{j,c})_j$, with $j$ ranging from $\ell +1$ to $\infty,$ is an $n$-periodic sequence of univalent maps.

\begin{remark} \normalfont \label{gjdjwew}
Note that $\ell \geq 1$. Indeed, the point $a$ cannot lie on the cycle $\alpha_a$ without forcing the critical point to belong to the same cycle, since $\mathbf{f}_a^{-1}(a) = \{0\}$.
\end{remark}

  For every $c$ sufficiently close to $a,$ the periodic sequence \eqref{oaoijsmdfsebdf}  determines a \emph{repelling cycle $\alpha(c)$} of period $n.$ It follows that
\begin{equation}\label{alsokkkdfoiwegasc} g_c(z) = f_{\ell,c} \circ \cdots \circ f_{1,c}(z)
\end{equation}  is a holomorphic family in the sense that $(c, z)\mapsto g_c(z)$ is holomorphic on a neighborhood of $(a,a)$, with $g_a$ sending $a$ to $z_\ell.$

\begin{lem}\label{gfgdgew}  The map $(c, z) \mapsto g_c(z)$ in \eqref{alsokkkdfoiwegasc} is holomorphic on a neighborhood $V_a\times V_a$ of $(a,a).$ We may choose $V_a$ so that $g_c$ is univalent on $V_a,$ whenever $c\in V_a$. For every $c$ sufficiently close to $a$, $g_c(V_a)$ contains $\tilde{z}_{\ell}(c)$,  and $g_c$ sends a unique point $\xi(c)$  of $V_a$ to $\tilde{z}_\ell(c).$ The map $c \mapsto \xi(c)$ is holomorphic on a neighborhood of $a$ and $\xi(a)=a.$ Define $z_j(c)$ inductively by setting $z_0(c)= \xi(c)$ and $z_j(c)=f_{j,c}(z_{j-1}(c)).$ Then
\begin{equation}\label{alkjsoidfsqcs}
\xi(c) \xrightarrow{f_{1,c}} z_1(c) \xrightarrow{f_{2,c}} z_{2}(c) \to \cdots \xrightarrow{f_{\ell, c}} z_{\ell}(c) \to \cdots
\end{equation}
is the unique bounded forward orbit of $\xi(c)$ under $\mathbf{f}_c,$ for every $c$ in a neighborhood of $a.$ Moreover,  $(z_j(c))_\ell^{\infty}$ is $n$-periodic and coincides with the repelling cycle $\alpha(c).$ 
For parameters $c$ in a neighborhood of $a,$ the multiplier of $\lambda(c)$ of the cycle $\alpha(c)$  is a holomorphic function of  $c$. 
\end{lem}

\begin{proof} By \eqref{alsokkkdfoiwegasc}, it is easy to find $V_a$ such that $g_c$ is univalent on $V_a$, for every $c$ in $V_a.$ Moreover, $g_c(V_a)$ contains a neighborhood $W_\ell$ of $z_\ell$ for every $c$ sufficiently close to $a.$ Since $\tilde{z}_\ell(c)$ belongs to $W_{\ell}$ for parameters close to $a$, it follows that $g_c(V_a)$ contains $\tilde{z}_\ell(c)$, for every $c$ in a neighborhood of $a.$ We may define $\xi(c)$ by  $g_c^ {-1}(\tilde{z}_{\ell}(c))$, and it becomes clear that $c\mapsto \xi(c)$ is holomorphic on a neighborhood of $a$, with $\xi(a) =a.$

According to the definition of $z_j(c)$, the map $g_c$ sends $\xi(c)$ to $z_\ell(c)$, which implies $z_{\ell}(c) = \tilde{z}_\ell(c).$ Using Lemma \ref{ghdcsdffewdf} and induction, we show that if $z_j(c)$ coincides with $\tilde{z}_j(c),$ with $j\geq \ell$,  then  $$z_{j+1}(c) = f_{j+1,c}(\tilde{z}_j(c)) = \tilde{z}_{j+1}(c).$$ Hence $z_j(c) = \tilde{z}_j(c)$, for every $j\geq \ell.$  
To complete the proof, we only need to show that, for $c$ sufficiently close to $a$,
\eqref{alkjsoidfsqcs} is the only bounded orbit of $\xi(c)$. In fact,   we will show that there exists a neighborhood $\check{V}_a$ of $a$ such that
\begin{equation}\label{lskdjfpoijcavsdfwd}
\mathbf{f}_c(z_j(c)) \cap K_c = \{ z_{j+1}(c)\}
\end{equation}
for every $c$ in $\check{V}_a$ and every $j\geq 0.$ In view of \eqref{safadfwdgasdvcqeqwsdfs}, there exists an escaping radius $R>0$ that applies for every  $\mathbf{f}_c$ with $c$ sufficiently close to $a.$  Since  \eqref{alkjsokjdkkdf} is the only bounded forward orbit of $z_0=a$ under $\mathbf{f}_a$ and since $z_j(a)=z_j,$ the equation \eqref{lskdjfpoijcavsdfwd} holds for $c=a,$ and every $j\geq 0$, otherwise we would have two bounded orbits of $a$, which is impossible. Notice that $\mathbf{f}_a(z_j)$ consists of $q$ distinct points, and only one of them belongs to $K_a$, namely, $z_{j+1}.$ The other $q-1$ points are in the basin of infinity $\mathbb{C}{\setminus} K_a$, and therefore some iterate $\mathbf{f}_a^N$ of the correspondence send all these $q-1$ points to the forward invariant set $\mathbb{C}{\setminus}\mathcal{B}_R$  defined by \eqref{safadfwdgasdvcqeqwsdfs}. The set function $(c,z) \mapsto \mathbf{f}_c(z)$ is continuous in the Hausdorff topology, and therefore, for a small perturbation $(c,z)$ of $(a, z_j)$ we conclude that $\mathbf{f}_c(z)$ consists of $q$  points which are very close to the points of $\mathbf{f}_a(z_j)$, so that $q-1$ points of $\mathbf{f}_c(z)$ are sent to $\mathbb{C}{\setminus}\mathcal{B}_R$   by $\mathbf{f}_c^N$, leaving only one point of $\mathbf{f}_c(z)$ very close to $z_{j+1}$ that might belong $K_c.$ If $c$ is sufficiently close to $a$, then we apply this analysis to $\mathbf{f}_c(z_j(c))$, and the conclusion is that this image set consists of $q$ points, with at most one point in $K_c$; but $z_{j+1}(c)$ is in the image set, and it belongs to $K_c$ because it is part of the bounded orbit  \eqref{alkjsoidfsqcs}. Thus, for each $j\geq 0$, equation  \eqref{lskdjfpoijcavsdfwd} holds for every $c$ in a neighborhood $V_j$ of $a$. Since \eqref{alkjsoidfsqcs} is pre-periodic, $V_j$ is also pre-periodic, and  we may consider the finite intersection of all such sets $V_j$ with $j$ ranging from $0$ until $\ell+n.$ Denote this intersection by $\check{V}_a.$ It follows that \eqref{lskdjfpoijcavsdfwd} holds for every $c$ in $\check{V}_a$ and every $j\geq 0.$ 

Notice that if $(w_j)_0^\infty$ is a bounded orbit of $w_0=z_0(c)$, with $c$ in $\check{V}_a$, then whenever $w_j=z_j(c)$ for $0\leq j \leq n$, the next point $w_{n+1}$ belongs to $\mathbf{f}_c(z_n(c))$ and also belongs to $K_c$, since $w_{n+1}$ is part of a bounded orbit. From \eqref{lskdjfpoijcavsdfwd} we conclude that $w_{n+1} = z_{n+1}(c).$ The uniqueness of the bounded orbit \eqref{alkjsoidfsqcs}  is established by this induction process. 
\end{proof}

Suppose that $a \in \mathbb{C}$ is a Misiurewicz point. We already know that  for every $c$ in a neighborhood of $a$, the map $g_c$ defined by \eqref{alsokkkdfoiwegasc} sends $\xi(c)$ to the first point $z_{\ell}(c)=\tilde{z}_\ell(c)$ of the associated repelling cycle $\alpha(c)$ in \eqref{oaoijsmdfsebdf}.  Using the same maps $f_{j, c}$ that appear in \eqref{oaoijsmdfsebdf} and \eqref{alkjsoidfsqcs}, we define
\begin{equation} \label{afsdfoijlkjblsdf} h_c(z) = f_{\ell +n, c} \circ \cdots \circ f_{\ell +1, c}(z).
\end{equation}
Notice from 
\eqref{asdpoioasdfpowikdf} that  $h_c=h_{\ell,c}.$
By Lemma \ref{gfgdgew}, $z_{\ell}(c)$ is a repelling fixed point of $h_c$, for every $c$ in a neighborhood of $a.$ We will use the maps $h_c$ and $g_c$ to define the concept of transversality, which is a key ingredient for establishing the asymptotic similarity between the Multibrot set and the Julia set at Misiurewicz points.

\begin{defi}[\bf Transversality condition] \label{asdljhafspdoifasdf}\normalfont The family  $\mathbf{f}_c$, given by \eqref{lkjdllkjhalsoed}, is said to satisfy
 the \emph{transversality condition} at a Misiurewicz parameter $a\in \mathbb{C}$ if the derivative of \[ w(c) = h_c(g_c(c)) - g_c(c)\] is nonzero at $c=a.$
\end{defi}

\begin{thm}[\bf Transversality of quadratic correspondences]  \label{dfasdfefd}  The family $\mathbf{f}_c$, given by \eqref{lkjdllkjhalsoed} with \((p, q) = (4, 2),\) satisfies the transversality condition at every Misiurewicz point.\end{thm}
We shall give an algebraic proof of this theorem in Section 
\ref{jgc}.

\section{Self-similarity for the Julia set}


Suppose that $a\in \mathbb{C}$ is a Misiurewicz point of the family \eqref{lkjdllkjhalsoed}. Let $(z_j(c))_0^\infty$ be the associated pre-periodic orbit of $z_0(c)=\xi(c)$,  as in \eqref{alkjsokjdkkdf} and \eqref{alkjsoidfsqcs}, with a pre-periodic sequence $(f_{j,c})_1^{\infty}$ of univalent branches of $\mathbf{f}_c,$ where each $f_{j,c}$ sends $ z_{j-1}(c)$ to $z_j(c).$ Recall that $\alpha(c)$ is the repelling cycle $(z_j(c))_{\ell}^{\ell +n}.$  We know from \eqref{afsdfoijlkjblsdf} that $h_c$, given by the composition of the $f_{j,c}$ along the cycle $\alpha(c)$, has a repelling fixed point at $z_{\ell}(c).$

 \begin{remark} \normalfont \label{alkjsdfoijcllksdfc}Following the  notation established in Lemma \ref{gfgdgew}, the  multiplier $h_c'(z_{\ell}(c))$ is denoted by $\lambda(c).$ As a consequence of the K\oe{}nigs Linearization Theorem, for every $c$ in a neighborhood $\mathcal{N}_a$ of the Misiurewicz point $a$,  there exists a unique univalent map $\varphi_c$ defined on a small conformal disk $V_c$ containing $z_{\ell}(c)$, with $\varphi_c(z_\ell(c))=0$ and $\varphi_c'(z_{\ell}(c)) =1$,  such that $U_c=h_c^{-1}(V_c)$ is compactly contained in $V_c$ and 
\begin{equation} \label{lkajlksdfscc} \lambda(c) \varphi_c(z)=\varphi_c(h_c(z)), \quad z \in U_c. 
\end{equation}
\noindent  Let $\lambda=\lambda(c).$ Since $z\mapsto \lambda z$ maps $\mathbb{D}_r$ onto $\mathbb{D}_{\lambda r},$ we may replace $U_c$ and $V_c$ by 
$\varphi_c^{-1}(\mathbb{D}_r)$ and $\varphi_c^{-1}(\mathbb{D}_{\lambda r})$, respectively, for any $r>0$ sufficiently small. Under this convention, the sets $U_c$ and $V_c$, now denoted  $U_{c,r}$ and $V_{c,r},$ are parameterized by $r$ and shrink to $z_{\ell}$(c) as $r\to 0.$ We may also assume that $\varphi_c$ is defined on an open set containing the closure of all such  $V_{c,r}.$ 

\end{remark}

\begin{thm}[\bf Asymptotic similarity for $K_c$] \label{gdasoiu}   Suppose that $a\in \mathbb{C}$ is a Misiurewicz point for the family \eqref{lkjdllkjhalsoed}, and let $f_{j,c}, \lambda(c), z_j(c)$ and  $\varphi_c$   be as  in  Lemma \ref{gfgdgew} and Remark  \ref{alkjsdfoijcllksdfc}.  If $\overline{V} \subset \operatorname{dom}(\varphi_c)$ and $V$ is an open set containing $z_\ell(c)$, let $\{B_j(c)\}_{j=0}^{\infty}$ be the sequence of compact sets inductively defined by  
\begin{equation}\label{slkjspoipjlaskjdfwedf}B_{\ell}(c)=\varphi_c(\overline{V} \cap K_c), \ \ B_{j}(c)=f_{j, c}'(z_{j-1}(c)) {\cdot} B_{j-1}(c).   \end{equation}
 For all $c$ in a neighborhood of $a$,  the Julia set $K_c$ is asymptotically $\lambda(c)$-self-similar about each $z_j(c),$ with limit model $B_j(c).$

\end{thm}

\begin{proof} Using the terminology adopted in Remark  \ref{alkjsdfoijcllksdfc}, 
we will prove that if $(c,r)$ is sufficiently close to $(a, 0)$, with $r>0,$ then 
\begin{equation}\label{hgdsew}
h_c\left(U_{c,r} \cap K_c\right)=V_{c,r} \cap K_c.
\end{equation}  The map $h_c$ is a forward branch of $\mathbf{f}_c^{n}$ and $h_c(z_{\ell}(c)) \in K_c.$  By Lemma \ref{gfgdgew},  $z_{\ell}(c)$ lies in the unique bounded forward orbit of  $z_0(c)=\xi(c).$ It follows that $h_c(z_{\ell}(c))$ is the only point in the forward image $\mathbf{f}_c^n(z_{\ell}(c))$ that belongs to $K_c.$ Since the complement of $K_c$ is open, using the continuity of $z \mapsto \mathbf{f}_c^n(z)$ one can show that any perturbation $z$ of $z_\ell(c)$ produces an image set $\mathbf{f}_c^n(z)$ that is very close to $\mathbf{f}_c^{n}(z_\ell(c))$; as a result, at most one point of $\mathbf{f}_c^n(z)$ lies in $K_c,$ namely $h_c(z).$ However, every point of $K_c$ has at least one image under $\mathbf{f}_c^n$ that remains in $K_c$, as  follows from the definition of $K_c$. Hence,  for every $z \in K_c$ in a neighborhood of $z_{\ell}(c)$, $h_c(z)$ is the only point of $\mathbf{f}_c^n(z)$ that remains in $K_c.$ By decreasing $r$, if necessary, we can ensure that the diameters of $U_{c,r}$  and $V_{c,r}$ are sufficiently small. By the previous argument,  it follows that \[h_c(U_{c,r}\cap K_c) \subset K_c\cap V_{c, r}.\]  The reverse inclusion $V_{c, r}\cap K_c \subset h_c(U_{c,r}\cap K_c)$ follows from the backward invariance of $K_c=\mathbf{f}_c^{-n}(K_c).$ Now let
\(\tilde{B}_{\ell}(c)=\varphi_c\left(K_c \cap \overline{V_{c,r}}\right).
\)
 Since $U_{c,r} = \varphi_c^{-1}(\mathbb{D}_r),$ it follows from \eqref{lkajlksdfscc} and  \eqref{hgdsew}   that
 \begin{equation} \label{alhasdfoapivasdfsdfc}
 \begin{split}
 \lambda(c) (\tilde{B}_{\ell}(c) \cap \mathbb{D}_r) &=  \lambda(c) (\varphi_c(\overline{V}_{c,r} \cap K_c)\cap \mathbb{D}_r ) \\  
 & =   \lambda(c) (\varphi_c(U_{c,r} \cap K_c) \cap \mathbb{D}_r)\\
 & = \varphi_c(V_{c,r} \cap K_c) \cap \mathbb{D}_{\lambda(c) r} \\ & = \varphi_c(\overline{V_{c,r}} \cap K_c) \cap \mathbb{D}_{\lambda(c) r} 
   = \tilde{B}_{\ell}(c) \cap \mathbb{D}_{\lambda(c) r}.
  \end{split}
 \end{equation}
\noindent Hence
\[ (\lambda(c) \tilde{B}_{\ell}(c)) \cap \mathbb{D}_{\lambda(c) r}=  \lambda(c)(\tilde{B}_{\ell}(c) \cap \mathbb{D}_r) = \tilde{B}_{\ell}(c) \cap \mathbb{D}_{\lambda(c) r}.\] By intersecting with $\mathbb{D}_r,$ it follows that
\(
(\lambda(c) \tilde{B}_{\ell}(c))_r=(\tilde{B}_{\ell}(c))_r.
\)
 Thus $\tilde{B}_{\ell}(c)$ is $\lambda(c)$-self-similar about $0.$ By Theorem \ref{gjdqwefd},  $K_c$  is asymptotically $\lambda(c)$-self-similar about $z_{\ell}(c)$, with limit model  $\tilde{B}_{\ell}(c),$ for then $\varphi_c'(z_\ell(c))=1.$  If $V$ is an open set containing $z_{\ell}(c)$ and $\overline{V} \subset \operatorname{dom}(\varphi_c)$, let \[B_{\ell}(c) =  \varphi_c(\overline{V} \cap K_c).     \]  Then $(\tilde{B}_{\ell}(c))_s = (B_{\ell}(c))_s,$ for some $s\in (0, r).$       It follows that $B_{\ell}(c)$ is also a limit model about $z_{\ell}(c).$

 A neighborhood $W_{j, c}$ of $z_j(c)$ is mapped onto a neighborhood $W_{j+1,c}$ of $z_{j+1}(c)$ by the univalent branch $f_{j+1,c}$  of $\mathbf{f}_c.$ As in the proof of \eqref{hgdsew}, after decreasing the diameter of $W_{j,c}$ if necessary, one can show  that 
 \[
 f_{j+1,c}(K_c\cap W_{j,c}) = K_c\cap W_{j+1,c}.
 \]
 Theorem \ref{gjdqwefd} implies that $K_c$ is asymptotically $\lambda(c)$-self-similar about each $z_j(c)$, from which \eqref{slkjspoipjlaskjdfwedf} follows inductively.
  \end{proof}

\section{Self-similarity for the Multibrot set}

We begin this section by recalling a well-known result due to Tan Lei.

\begin{thm}[\bf TAN Lei, 1990] \label{gkdjwerdfl}  Let $u$ and $\lambda$ be holomorphic functions defined on a neighborhood $U$ of $a \in \mathbb{C}$ with $u'(a) \neq 0,$ $u(a)=0,$ and $|\lambda(a)|>1.$ Suppose that $X$ is a  closed subset of $U \times \mathbb{C}$ such that

\begin{enumerate}[label=(\roman*)]
\item   for every $c$ in a neighborhood of $a$ contained in $U$,  \[0\in X(c) =\{x: (c,x) \in X \}\] and $X(c)$  is $\lambda(c)$-self-similar about zero with the same $r>0$:  \[(\lambda(c) X(c))_r = (X(c))_r;\] 
\item  there exists a dense subset $X'(a) \subset X(a)$ such that, for every $x\in X'(a)$, there exists a holomorphic function $\zeta_x$ defined on a neighborhood $V_x\subset U$ of $a$ such that  $\zeta_x(a) = x$ and $\zeta_x(c) \in X(c),$ for every $c\in V_x.$

\end{enumerate}
Under hypotheses (i) and (ii), the set   $$M_u=\{c\in U: u(c) \in X(c)\}$$ is asymptotically $\lambda(a)$-self-similar about $a,$ with limit model given by $u'(a)^{-1}{\cdot}X(a).$
\end{thm}
\begin{proof}[References] This corresponds to Proposition 4.1 on page 601 of \cite{TanLei}. The statement has been simplified to suit our purposes, omitting certain hypotheses that are unnecessary in our setting. For instance, assuming that both $u$ and $\lambda$ are  holomorphic ensures that conditions $(3)$ and $(4)$  of Proposition 4.1 of \cite{TanLei} are automatically satisfied. 
\end{proof}

  \begin{defi}[\bf Multibrot set]\label{lkjasdfasdfqefd}\normalfont In the parameter space, the set $M_{p,q}$ consists of all  $c \in \mathbb{C}$ for which $0 \in K(\mathbf{f}_c)$, where $\mathbf{f}_c$ is the family given by \eqref{lkjdllkjhalsoed}. 
  \end{defi}

  This set generalizes the classical Mandelbrot set associated with quadratic polynomials.

If $a \in \mathbb{C}$ is a Misiurewicz parameter, then \cite[Theorem D]{siqueira2025suhyperbolic}  ensures that $K_a = J_a$  and 
\[
a \in M_{p,q} \cap J_a.
\]
A remarkable similarity between $M_{p,q}$ and $J_a$ emerges at small scales around the point $a$:

\begin{thm}[\bf Similarity between  the Multibrot  and  Julia sets] \label{gaeiogc} Suppose that $a \in \mathbb{C}$ is a Misiurewicz parameter for the family \eqref{lkjdllkjhalsoed}. Assume further that the transversality condition holds at $a$. Then both the Julia set $J_a$ and the Multibrot set $M_{p,q}$ are asymptotically self-similar about $a$, with the same scale \(\lambda=\lambda(a) \) made explicit in Theorem \ref{gdasoiu}. Their respective limit models coincide up to multiplication by a nonzero complex constant. More precisely, the limit model about $a \in M_{p,q}$ is $\mu_a B_{0}(a)$, where $B_{0}(a)$ is the limit model about $a\in K_a$ presented in Theorem~\ref{gdasoiu} and  $\mu_a$ is defined by \eqref{lkslkjsdfoijwlkjjjkdfwd}.

\end{thm}

The proof will be presented after some preparatory lemmas. Since the transversality condition holds for quadratic correspondences (see Theorem~\ref{dfasdfefd}), we obtain the following immediate consequence.

\begin{cor}[\bf Similarity for quadratic correspondences] \label{adflhlkjhasdlkjhfqewfdasdfc}Let $a$ be a Misiurewicz point for the family $\mathbf{f}_c$, given by \eqref{lkjdllkjhalsoed} with \((p, q) = (4, 2).\) Then $J_a$ and $M_{4,2}$ are asymptotic self-similar about $a$ with the same scale,  and the same limit model  up to multiplication by $\mu_a.$ 
\end{cor}

In the following lemma (and its proof) we use the same notation as in Lemma~\ref{gfgdgew}, equation \eqref{afsdfoijlkjblsdf} and Remark~\ref{alkjsdfoijcllksdfc}.

\begin{lem}\label{ggvbfheq} The function

$$ \Phi(c,z) = (c, \varphi_c \circ g_c(z))$$ 

\noindent is well-defined and holomorphic on a neighborhood of $(a,a), $ with $\Phi(a,a) = (a,0).$

\end{lem}

\begin{proof} It is well known that the K\oe{}nigs function $\varphi_c$ depends holomorphically on $c$ -- for more details, see  \cite[p. 78]{Milnor}.   Since $g_c$ sends $\xi(c)$ to $z_{\ell}(c)$, with $\xi(a)=a,$ it follows that $\Phi(a, a)= (a, 0),$ and $\Phi$  must be defined on a neighborhood of $(a, a).$ \end{proof}

\begin{lem} \label{gsdghcxew}Let $\Phi,$ $h_c$ and $\mathcal{N}_a$ be as  in Lemma \ref{ggvbfheq}, equation \eqref{afsdfoijlkjblsdf} and Remark \ref{alkjsdfoijcllksdfc}. Then $\Phi$  establishes a diffeomorphism from a neighborhood $\tilde{U}$ of $(a, a)$ onto a neighborhood $\Phi(\tilde{U})$  of $(a, 0).$ Let $W\times \overline{\mathbb{D}}_r$ be a closed neighborhood of $(a,0)$ contained in  $\Phi(\tilde{U})$, where $W\subset \mathcal{N}_a$ is conformal closed disk containing $a$ in its interior. We may choose $W$ and $r>0$ sufficiently small so that: 
\begin{enumerate}[label=(\roman*)]
\item $\Phi(c,c) \in W \times \mathbb{D}_r$, for every $c\in W;$ \item $\Omega = \Phi^{-1}(W\times \overline{\mathbb{D}}_r)$   is compact; \item   for every $c\in W,$ $h_c(z_{\ell}(c))=z_{\ell}(c);$ moreover, the section 
\[\Omega_c= \{ z \in \mathbb{C}: (c,z) \in \Omega\} \]
is conformally isomorphic to $\overline{\mathbb{D}}$ and its interior contains $c$ and $\xi(c)$, where $\xi$ is the same map made explicit in Lemma \ref{gfgdgew};
\item if $c\in W,$ then $\Omega_c$ is contained in the domain of $g_c$ and  \begin{equation}\label{gadgsewer}g_c(K_c \cap \Omega_c) =K_c \cap g_c(\Omega_c).\end{equation}

\end{enumerate} 
\end{lem}

\begin{proof} Using the Inverse Function Theorem and 
\[ \Phi(a, a) = (a, \varphi_a \circ g_a(a))= (a, \varphi_a(z_{\ell}))=(a, 0)
\] it is possible to show that $\Phi$ is diffeomorphism from a neighborhood of $(a, a)$ onto a neighborhood of $(a, 0)$. In particular, $\Phi^{-1}$ is well-defined on some \[\Phi(\tilde{U}) \supset W\times \overline{\mathbb{D}}_r\] where $W$ is a conformal closed disk whose interior contains  $a,$ so that   $\Omega=\Phi^{-1}(W\times \overline{\mathbb{D}}_r)$ is compact. This proves (ii). Since $\varphi_c \circ g_c(z)$ depends holomorphically on $(c, z)$ and sends $(a, a)$ to zero, it is natural to assume that $\varphi_c \circ g_c(c)$ lies in $\mathbb{D}_s$, for some positive $s< r$, for all $c\in W$, thereby proving (i). 
Note that 
\begin{equation}\label{llkjhasdfoiqdf}\Omega_c = (\varphi_c\circ g_c)^{-1}(\overline{\mathbb{D}}_r)
\end{equation} for every $c\in W.$ Then $\xi(c)$ corresponds to the inverse image of zero; in particular, $\xi(c)$ is in the interior of $\Omega_c$, for any $c\in W.$ From (i) it follows that $\Phi(c, z) \in W \times \mathbb{D}_r$, for any $c\in W$ and for any perturbation $z$ of $c.$ Hence $c$ is also in the interior of $\Omega_c$, whenever $c\in W.$ From \eqref{llkjhasdfoiqdf} we conclude that $\Omega_c$ is conformally isomorphic to $\overline{\mathbb{D}}.$ This proves (iii).

Now we proceed to the proof of (iv), which is the most delicate step. Recall from \eqref{alkjsokjdkkdf} that $(z_i)_0^{\infty}$ is the only bounded forward orbit of $z_0=a$ under $\mathbf{f}_a$, and that $z_{\ell}$ is the first repelling periodic point in this orbit.  Since $a\neq 0,$ the action of $\mathbf{f}_a$ on a sufficiently small neighborhood of $a$ is determined by $q$ distinct univalent branches with pairwise disjoint images, and only one of them intersects $K_a$, otherwise the Misiurewicz point $a$ would have at least two bounded orbits, which is impossible. By the same argument and using the fact that no iterate  $\mathbf{f}_a^k(0)$ contains zero, one can show that the action of $\mathbf{f}_a$ on a small neighborhood $U_{\zeta}$ of each point $\zeta$ in the image of $a$ under $\mathbf{f}_a$ is determined by $q$ univalent branches $\psi_j$ with pairwise disjoint images $\psi_j(U_\zeta)$, none of which intersects $K_a$ in the case where $\zeta \notin K_a$, and only one image $\psi_s(U_\zeta)$ intersects $K_a$ if $\zeta\in K_a.$  After repeating this argument $\ell$ times we find a set $\mathcal{F}_a$ consisting of $q^{\ell}$ univalent branches  $f_a$ defined on a very small neighborhood $\mathcal{U}_a$ of $a$; the images $f_a(\mathcal{U}_a)$, $f_a\in \mathcal{F}_a,$ are small conformal disks, not necessarily pairwise disjoint, but only one of them intersects $K_a.$  It is clear that any germ of holomorphic branch of $\mathbf{f}_a^{\ell}$ at $a$ is determined by one element of $\mathcal{F}_a.$

Recall from  \eqref{asdfasoiwellkjsdf} that every univalent branch $h$ of $\mathbf{f}_a$ can be perturbed to produce a holomorphic family $(c, z) \mapsto h_c(z)$   such that $h_a =h$ and $h_c$ is a univalent branch of $\mathbf{f}_c$, with $c$ sufficiently close to $a.$  Since any
$f_a$ in $\mathcal{F}_a$ is a composition of univalent branches of $\mathbf{f}_a$, it is possible to show that any  $f_a
\in \mathcal{F}_a$ gives rise to a holomorphic family $(c, z) \mapsto f_c(z)$ defined on a neighborhood $\mathcal{U}_a\times \mathcal{U}_a$ of $(a, a)$ such that, for any $c$ in $\mathcal{U}_a,$ the map  $f_c:\mathcal{U}_a \to \mathbb{C}$ is a univalent branch of $\mathbf{f}_c^{\ell}.$ Since $\mathcal{F}_a$ is finite, we may assume that the domain of every  holomorphic family $(c, z) \mapsto f_c(z)$ is the same set $\mathcal{U}_a\times \mathcal{U}_a.$
 Let $\mathcal{F}_c$ denote the set of all $f_c$ obtained in this way.  $\mathcal{F}_c$ is a perturbation of $\mathcal{F}_a$ in the sense that $c\mapsto f_c(A)$ is a continuous function of the parameter $c\in \mathcal{U}_a$ (Hausdorff topology),  whenever $f_a\in \mathcal{F}_a$ and  $A$ is a nonempty compact subset of  $\mathcal{U}_a.$

It follows from \eqref{alsokkkdfoiwegasc} and    Lemma \ref{gfgdgew}  that $g_a$ sends $a$ to $z_{\ell}$ and $g_c \in \mathcal{F}_c,$ for all $c\in \mathcal{U}_a.$ 
As we have seen, $g_a(\mathcal{U}_a)$ is the only set in $\{ f_a(\mathcal{U}_a): f_a\in \mathcal{F}_a\}$ which intersects $K_a. $   By reducing $\mathcal{U}_a$ if necessary, the perturbed family $\mathcal{F}_c$ satisfies an analogous property: if $c\in \mathcal{U}_a,$ $f_c\in \mathcal{F}_c$ and $f_c(\mathcal{U}_a)$ intersects $K_a$, then $f_c=g_c.$ All other sets $f_c(\mathcal{U}_a),$ $f_c\neq g_c$,  are contained in the basin of infinity $\mathbb{C} {\setminus} K_a.$ The escaping radius is characterized by the condition in equation~\eqref{safadfwdgasdvcqeqwsdfs}. After further shrinking $\mathcal{U}_a$, we choose an escaping radius $R > 0$ that works uniformly for every correspondence $\mathbf{f}_c$ with $c \in \mathcal{U}_a$. In particular, \eqref{safadfwdgasdvcqeqwsdfs} holds for all $c \in \mathcal{U}_a$.

 Let $f_a\neq g_a$ be in $\mathcal{F}_a.$ Then $f_a(\mathcal{U}_a)$ is a small conformal disk contained in $\mathbb{C} {\setminus} K_a$. Fix $z_0$ in $f_a(\mathcal{U}_a).$ By Theorem \ref{adsfasdgglkppsdfsdf}, there exists $k>0$ such that $\mathbf{f}^{k}_a(z_0) \subset \mathbb{C}{\setminus}\mathcal{B}_R$, where $\mathcal{B}_R$ is defined in \eqref{safadfwdgasdvcqeqwsdfs}.  Then $\mathbf{f}_c^{k}(z) \subset \mathbb{C}{\setminus}\mathcal{B}_R$, if $(c, z)$ is sufficiently close to $(a, z_0).$ A priori, $k$ depends on $f_a$, but since we have finitely many maps, we may assume that  $\mathcal{U}_a$ is sufficiently small so that \begin{equation}\label{adfasdfqwkkslkdf}\mathbf{f}_c^{k}(f_c(\mathcal{U}_a)) \subset \mathbb{C}{\setminus}\mathcal{B}_R\end{equation}
 
 \noindent whenever $c\in \mathcal{U}_a$ and $f_c\neq g_c.$  We conclude that \( f_c(\mathcal{U}_a) \subset \mathbb{C} {\setminus} K_c \) for all \( c \in \mathcal{U}_a \), except in the case where \( f_c = g_c \), in which \( g_c(\mathcal{U}_a) \) may intersect \( K_c \).
Suppose that \( c \in \mathcal{U}_a \). If \( z \in K_c \cap \mathcal{U}_a \), then there exists at least one image $w$ of \( z \) under \( \mathbf{f}_c^{\ell} \) that lies in \( K_c \).
 Based on the previous argument,  the only possibility is $w=g_c(z),$ which shows that \begin{equation}\label{asdfallkspodgsg}g_c(K_c \cap \mathcal{U}_a)  \subset K_c \    \  (c\in \mathcal{U}_a).\end{equation}

Recall that $\xi(c)$ belongs to $\Omega_c$ and $\xi(c)$ converges to  $a$ as $c\to a.$  Moreover, for $c\in \mathcal{U}_a$, the diameter of the set $\Omega_c = g_c^{-1}\circ \varphi_c^{-1} (\mathbb{D}_r)$ goes to zero uniformly as $r\to 0.$ We conclude that $\Omega_c$ is contained in $\mathcal{U}_a$, provided $r>0$ is sufficiently small and $c$ is in a small neighborhood $\mathcal{V}_a \subset \mathcal{U}_a$ of $a.$ It follows from \eqref{asdfallkspodgsg} that $$g_c(K_c \cap \Omega_c) \subset g_c(K_c\cap \mathcal{U}_a) \subset K_c $$ whenever $c\in \mathcal{V}_a.$  The reverse inclusion in~\eqref{gadgsewer} follows directly from the backward invariance of \( K_c \) under \( \mathbf{f}_c \).
If we take $W=\mathcal{V}_a$, then \eqref{gadgsewer} holds for every $c\in W.$ 
\end{proof}

\begin{lem} \label{ggsfgccwe} Let $K$ denote the set of all $(c,z)$ in $\mathbb{C}^2$ such that $z\in K_c.$
Then $K$ is closed. 

\end{lem}

\begin{proof}  Consider the multifunction $F(c,z) = (c, \mathbf{f}_c(z))$ defined on $\mathbb{C}^2.$  Let $R_c$ denote an escaping radius of $\mathbf{f}_c$, as in  equation~\eqref{safadfwdgasdvcqeqwsdfs}. Let $V_{k}$ be the set of all $(c,z)$ such that $|c| \leq k$ and $|z|\leq R_c.$

 We will prove that $\{F^{-n}(V_k)\}_n$ is a nested sequence of compact sets whose intersection is $K \cap \{ |c| \leq k\}{\times} \mathbb{C}$.   By induction, 
 \[F^{-n}(V_k) = \{ (c,z): z \in \mathbf{f}_c^{-n}(|w| \leq R_c), \ |c| \leq k\}. \]

Hence the sets are nested, bounded, and (by \eqref{alhasdfoilkhlkjhasdfc}) their intersection  is $K \cap \{ |c| \leq k\}{\times}\mathbb{C}.$ If $|c| \leq k$ and $(c,z)$ is in the complement of $F^{-n}(V_k)$  then every forward orbit 
\[ z \xrightarrow{\mathbf{f}_c} z_1 \xrightarrow{\mathbf{f}_c} z_2 \xrightarrow{\mathbf{f}_c} \cdots \xrightarrow{\mathbf{f}_c} z_n \]

terminates  at point with $|z_n| > R_c.$ Since $(c, z) \mapsto \mathbf{f}_c(z)$ is a continuous multifunction, this inequality  is persistent under small perturbations of the initial point $z$ and the parameter $c.$ Hence the complement of $F^{-n}(V_k)$ is open. Since the closed set $F^{-n}(V_k)$ is bounded, it must be compact.

Since the intersection of $K$ with every strip $|c| \leq k$ is compact, it follows that $K$ is closed.  
\end{proof}

\begin{lem} \label{gdsdgeweqdfb} Let \( \lambda(c) \), \( g_c \), \( \varphi_c \), \( W \), \(\mathbb{D}_r \),  \( K \), and \( \Omega \) be as defined in Remark~\ref{alkjsdfoijcllksdfc} and in Lemmas~\ref{gfgdgew}, \ref{gsdghcxew}, and~\ref{ggsfgccwe}. For every $c\in W$, define 
\begin{equation}\label{gdsdcvscswer}
\begin{split}
Y(c)  = \varphi_c \circ g_c(K_c \cap \Omega_c) & \qquad X(c)  = Y(c) \cup \partial \mathbb{D}_r\\
Y  = \bigcup_{c\in W} \{c\} \times Y(c) \ \  & \qquad \quad \,
X  = \bigcup_{c\in W} \{c\} \times X(c).
\end{split}
\end{equation}
\noindent Then $Y = \Phi(K \cap \Omega)$ is compact and $X$ is closed in $\mathbb{C}^2.$ For every $c\in W,$ both $X(c)$ and $Y(c)$ are $\lambda(c)$-self-similar about zero.

  \end{lem}

\begin{proof}  By Lemmas \ref{ggsfgccwe}  and  \ref{gsdghcxew},  the set $\Phi(K\cap \Omega)$ is closed and is given by all $(c, \varphi_c\circ g_c(z))$ such that $z$ is in $K_c\cap \Omega_c$ and $c\in W.$  Hence $Y=\Phi(K\cap \Omega).$ 
By Theorem \ref{gdasoiu}, $K_c$  is asymptotic $\lambda(c)$-self-similar about $z_{\ell}(c)$  with limit model $\varphi_c(K_c \cap \overline{V}),$ for   any open set $V$ containing $z_{\ell}(c)$ such that $ \overline{V} \subset \operatorname{dom}(\varphi_c).$  By Lemma \ref{gsdghcxew}, we may take $\overline{V}=g_c(\Omega_c)$, for then $\xi(c)$ is in the interior of $\Omega_c$ and therefore $z_{\ell}(c) = g_c(\xi(c))$ is in the interior of $g_c(\Omega_c).$ Hence $Y(c)$ is a limit model about $z_{\ell}(c)\in K_c$, and as a result, it must be  $\lambda(c)$-self-similar about zero.  The same  is true for $X(c).$ It is easy to show that $X$ is closed. 
\end{proof}

\begin{lem} \label{djeowek} Let \(g_c\), \(\varphi_c\), \(\Omega\), \(W\), \( X \) and \(M_{p,q}\) be as in  Lemma \ref{gdsdgeweqdfb} and Definition \ref{lkjasdfasdfqefd}. For any $c\in W$, $g_c(c)$ is in the domain of $\varphi_c$ and  \[u(c) = \varphi_c(g_c(c))\] defines a holomorphic function on $W.$ Let $M_{u}$ be the set of all $c \in W$ such that $u(c) \in X(c).$ Then \begin{equation} \label{gdadgeeddsqe} M_{u} = M_{p,q} \cap W.\end{equation} 

\end{lem}
\begin{proof}  Due to Lemma \ref{gsdghcxew}, $\Phi(c,c) =(c, u(c))$ is well-defined and belongs to $W\times \mathbb{D}_r$, for every $c$ in $W$. Hence  $u$ is a holomorphic function $W \to \mathbb{D}_r.$ In particular, $u(c) \notin \partial \mathbb{D}_r$ as $c\in W,$ so that $M_u$ is the set of all $c\in W$ such that $\varphi_c(g_c(c)) \in \varphi_c (g_c(K_c \cap \Omega_c)).$ Since $\varphi_c$ is univalent, $M_u$   must be contained in $\{c\in W: g_c(c) \in g_c(K_c\cap \Omega_c)\}.$  Since $g_c$ is univalent and $c\in \Omega_c$ as $c\in W$ (see Lemma \ref{gsdghcxew}), this implies
$$M_u \subset \{c\in W: c \in K_c \cap \Omega_c \} = \{c\in W : c\in K_c \} = M_{p,q} \cap W. $$ A similar argument in the reverse direction shows that \( M_{p,q} \cap W \subset M_u \). The proof is complete.
\end{proof}

\begin{lem}\label{adfasdfabsdabasd} Let $a$ be a Misiurewicz parameter for the family \eqref{lkjdllkjhalsoed}. Let $X(c)$ be as in Lemma \ref{gdsdgeweqdfb}. There exists a dense subset $X'(a) \subset X(a)$ such that, for each $x\in X'(a)$, there exists a holomorphic function $\zeta_x$ on a neighborhood $V_x$ of $a$, with $\zeta_x(a) =x$ and $\zeta_x(c) \in X(c),$ for every $c\in V_x.$ 

\end{lem}

\begin{proof}
Let \( \mathcal{R}_a \) denote the set of all repelling periodic points of \( \mathbf{f}_a \). By Theorem \ref{fsdfawedsfasdeed}, this set is dense in \( J_a = K_a \). Recall from Lemma \ref{gsdghcxew} the main properties of \( \Omega_c \), and let \( \Omega_c^{\circ} \) denote its interior. Define
\[
X'(a) = \varphi_a \circ g_a\left( \mathcal{R}_a \cap \Omega_a^{\circ} \right) \cup \partial \mathbb{D}_r.
\]

\noindent We claim that the closure of \( X'(a) \) is \( X(a) \). Indeed, since \( X(a) \) is closed and contains \( X'(a) \), it follows that \( X(a) \supset \overline{X'(a)} \). Conversely, we show that every point \( \check{z} \in X(a) \) belongs to \( \overline{X'(a)} \). Without loss of generality, assume \( \check{z} \notin \partial \mathbb{D}_r \). Then \( \check{z} = \varphi_a \circ g_a(\check{w}) \), for some \( \check{w} \in K_a \cap \Omega_a \). Note that \( \check{w} \notin \partial \Omega_a \), since otherwise \( \check{z} \in \partial \mathbb{D}_r = \varphi_a \circ g_a(\partial \Omega_a) \). Therefore, \( \check{w} \in \Omega_a^{\circ} \).

Since \( \mathcal{R}_a \) is dense in \( K_a \), there exists a sequence \( w_j \in \mathcal{R}_a \) converging to \( \check{w} \). As \( \check{w} \in \Omega_a^{\circ} \), it follows that all but finitely many \( w_j \) lie in the interior of \( \Omega_a \). We then have
\[
z_j = \varphi_a \circ g_a(w_j) \to \check{z},
\]
with \( z_j \in X'(a) \), hence \( \check{z} \in \overline{X'(a)} \). This proves that \( X(a) = \overline{X'(a)} \). (See Remark \ref{asdfadacbbsdasb}.)

From \eqref{llkjhasdfoiqdf}, we know that \( c \mapsto \Omega_c \) is continuous at \( a \) in the Hausdorff topology. For every \( x \in X'(a) {\setminus} \partial \mathbb{D}_r \), there exists a unique repelling periodic point \( z_j \in \mathcal{R}_a \cap \Omega_a^{\circ} \) such that \( \varphi_a \circ g_a(z_j) = x \).

By Lemma \ref{ghdcsdffewdf}, the point \( z_j \) gives rise to a unique holomorphic function \( c \mapsto \tilde{z}_j(c) \) defined on a neighborhood of \( a \), with \( \tilde{z}_j(a) = z_j \). Define
\[
\zeta_x(c) = \varphi_c \circ g_c(\tilde{z}_j(c)).
\]
Then \( \zeta_x(c) \) is well-defined for every \( c \) near \( a \), and satisfies the desired properties.
If \( x \in \partial \mathbb{D}_r \), the situation is simpler: in this case, \( \zeta_x \) is simply a constant map.
\end{proof}

\begin{remark}\label{asdfadacbbsdasb} \normalfont Without further analysis of the structure of the Julia set \( K_a \) of the correspondence, it is not possible to show that \( \mathcal{R}_a \cap \Omega_a^{\circ} \) is dense in \( J_a \cap \Omega_a \).
For instance, if \( J_a \) is the union of the unit circle and one of its diameters, and \( \Omega_a \) is the closed unit disk (recall that \( \Omega_a \) is always a conformal closed disk intersecting the Julia set), then it is clear that the closure of \( \mathcal{R}_a \cap \Omega_a^{\circ} \) is not equal to \( J_a \cap \Omega_a \), since one set contains only a line segment while the other contains a full circle.
This example shows that, in the statement of Lemma \ref{adfasdfabsdabasd}, we cannot replace \( X(c) \) with the alternative natural candidate \( Y(c) \), which is introduced in Lemma \ref{gdsdgeweqdfb}. 
\end{remark}

We are in a position to prove Theorem \ref{gaeiogc}.

\begin{proof}[Proof of Theorem \ref{gaeiogc}] 
We will apply Theorem  \ref{gkdjwerdfl} for $u(c)=\varphi_c(g_c(c))$, which is defined on an open set $U\supset W,$ as in Lemma \ref{gsdghcxew}.    According to the previous lemmas, every hypothesis of Theorem \ref{gkdjwerdfl}  has already been proved except $u'(a)\neq 0.$

Let  $z_{\ell}$, $z_{\ell}(c)$, $g_c,$ $\varphi_c$ and $h_c$ be as in  \eqref{alkjsokjdkkdf},  Lemma \ref{gfgdgew}, Remark \ref{alkjsdfoijcllksdfc} and  \eqref{afsdfoijlkjblsdf}.  Define $\beta(c) = g_c(c)$ and $$w(c) = h_c(g_c(c)) - g_c(c),$$ for every $c\in W.$  Then $u(c)= F(c, \beta(c))$ where $F(c,z)=\varphi_c(z).$ (The partial derivatives of $F$ are going to be denoted by $F_c$ and $F_z.$) 
Since $F_z(a, z_{\ell}(a)) = \varphi_a'(z_{\ell})=1,$  
\begin{equation} \label{gdgsdocv} u'(a) = F_c(a, z_{\ell}) + F_z(a, z_{\ell}) \cdot \beta'(a) = F_c(a, z_{\ell}) + \beta'(a) . \end{equation}

\noindent By Remark \ref{alkjsdfoijcllksdfc} and Lemma \ref{gsdghcxew},  $F(c, z_{\ell}(c)) = 0$ and $h_c(z_{\ell}(c)) =z_{\ell}(c)$, for every $c\in W.$ Moreover,
\begin{equation}
 \begin{split}
F_c(a, z_{\ell}) &= \lim_{c \to a} \frac{F(c, z_{\ell}) - F(a, z_{\ell})}{c-a} \\
 & = \lim_{c\to a} \frac{F(c, z_{\ell}) - F(a, z_{\ell})}{c-a}  + \frac{F(a, z_{\ell}) - F(c, z_{\ell}(c))}{c-a}\\
 & = F_c(a, z_{\ell}) - (F_c(a, z_{\ell})  + F_z(a, z_{\ell}) \cdot z_{\ell}'(a))\\
 & = -F_z(a, z_{\ell}) \cdot z_{\ell}'(a)=-z_{\ell}'(a). 
 \end{split}
 \end{equation}

By \eqref{gdgsdocv}, $u'(a) = \beta'(a) - z_{\ell}'(a).$ 

Using the identity theorem we conclude that there exists a sequence $c_n \to a$ such that $\beta(c_n) \neq z_{\ell}(c_n)$, otherwise we would have $z_{\ell}(c) = \beta(c)$ for all  $c$ and, therefore, $w=0$ on $W$ (since $W$ is connected). Notice that $w'(a)\neq 0$, since the transversality holds by hypothesis.

From $w(c) = h_c(\beta(c)) - h_c(z_{\ell}(c)) + z_{\ell}(c) - \beta(c)$  and $h_a'(z_\ell)=\lambda(a)$ we have
\begin{equation}
 \begin{split}
 w'(a) &=\lim_{n\to \infty} \frac{w(c_n) - w(a)}{c_n-a} \\ &= \lim_{n\to \infty} \frac{\beta(c_n) - z_{\ell}(c_n)}{c_n -a} \left (\frac{h_{c_n}(\beta(c_n)) - h_{c_n}(z_{\ell}(c_n))}{\beta(c_n) - z_{\ell}(c_n)} -1 \right)  \\ &= (\beta'(a) - z_{\ell}'(a)) (\lambda(a) -1) = u'(a) (\lambda(a) -1).
 \end{split}
 \end{equation} 
 
To prove the third equality, we decompose $h_{c_n}$ as $L_{c_n} + G_{c_n}$, where $L_{c_n}$ denotes the linear part of the Taylor expansion at $a$, and $G_{c_n}$ is the remainder term satisfying $G_{c_n}(a) = 0$ and $G_{c_n}'(a) = 0$. Since the map $(c, z) \mapsto h_c(z)$ is holomorphic, the Weierstrass theorem on uniform convergence implies that $h_{c_n}' \to h_a'$ uniformly on compact subsets. Applying the mean value inequality to $G_{c_n}$ then shows that the quotient involving $h_{c_n}$ converges to $\lambda(a) = h_a'(z_{\ell})$.

We conclude that \[u'(a) = \frac{w'(a)}{(\lambda(a) -1) }\neq 0.\]

By Lemma \ref{djeowek} and Theorem \ref{gkdjwerdfl}, $M_{u} = M_{p, q} \cap W$ is asymptotic $\lambda(c)$-self-similar about $a$ and the limit model is $X(a)/u'(a).$  Recall from Theorem \ref{gdasoiu} that $K_a$ is asymptotic $\lambda(a)$-self-similar about $a$ and that, for any open set $V\subset \operatorname{dom}(\varphi_a)$ with $z_{\ell}\in V,$ the limit model about $a\in K_a$ is
\[B_{0}(a)= B_{\ell}(a)/g_a'(a)=\varphi_a(K_a \cap \overline{V})/g_a'(a).\]  
Using \eqref{gadgsewer} we conclude that $X(a)$ and $\varphi_a(K_a\cap \overline{V})$ coincide on a neighborhood of zero; thus $X(a)/g_a'(a)$ is the limit model of $K_a$ about $a.$ In the statement of  Theorem \ref{gaeiogc} we must
 take \begin{equation}\label{lkslkjsdfoijwlkjjjkdfwd}\mu_a = g_a'(a)/ u'(a).
 \end{equation} The proof is complete.
\end{proof}

\section{Proof of the transversality} \label{jgc}

The proof of Theorem~\ref{dfasdfefd} (Transversality) will be presented after a sequence of lemmas. Before that, we recall some well-known properties of valuations.

\subsection{Valuations.} \label{secA1}

For any field $K$, a valuation on $K$ is a function 
$v: K \to \mathbb{Z}\cup \{\infty \} $ satisfying the following properties:

\begin{enumerate}[label=(\roman*)]
 \item   $v(ab)=v(a) +v(b)$; \item   $v(a) =\infty$ if, and only if, $a=0;$  
\item $v(a +b) \geq \min \{v(a), \ v(b) \}$ and
\item $v(a +b) = \min \{v(a), \ v(b) \}$ provided $v(a)\neq v(b).$

\end{enumerate}

If $L/K$ is a field extension for which every element of $L$ is a root of some polynomial with coefficients in $K$, then we say that $L/K$ is an algebraic extension. The following result is well known:

\begin{thm} \label{gsdge} If $K$ is field with valuation $v$ and $L$ is an algebraic extension of $K$, then there exists an extension of $v$ to a valuation on $L.$ 

\end{thm}

A field $K$ is \emph{algebraically closed} if every $a \in K$ is a root of some  non-constant polynomial with coefficients in $K.$
Every field $K$ has an algebraic extension $L$ which is algebraically closed. Since every such extension is unique up to an isomorphism fixing the elements of $K$, $L$ is called \emph{the algebraic closure} of $K.$

A complex number $z$ which is a root of some monic polynomial with coefficients in $\mathbb{Z}$ is an \emph{algebraic integer.}  
We will define the $2$-adic valuation $v_{2}$ on $\mathbb{Q}.$ If $a$ is a non-zero integer, then $v_2(a)$ is the greatest $k \geq 0$ such that $2^k$ divides $a.$ If $a=0$, then we set $v_2(0)=\infty.$ If $r$ is a rational number and $r=p/q$, then $ v_2(p) - v_2(q)$ does not depend on the particular choice of $p$ and $q$. This difference is, by definition, the value of $v_2(r).$ It is possible to check that $v_2: \mathbb{Q} \to \mathbb{Z} \cup \{ \infty\}$ is a valuation, known as the \emph{$2$-adic valuation.}  By Theorem \ref{gsdge}, there exists an extension of $v_2$ to the algebraic closure $\overline{\mathbb{Q}}.$

 \begin{thm} If $z$ is an algebraic integer, then $v_2(z) \geq 0.$
 
 \end{thm}

\begin{proof}  Let $v=v_2.$ Since $z$ is an algebraic integer, we have $z^n = a_{n-1} z^{n-1} + \cdots + a_1z +a_0,$
where $n>0$ and $a_k\in \mathbb{Z}$ for every $k.$  It follows that  
$$ n v(z) = v(z^n) \geq \min_{0\leq j<n} v(a_jz^j) = v(a_kz^k) = kv(z) + v(a_k),$$ for some $k$ with $0 \leq k <n.$ Thus  $(n-k) v(z) \geq v(a_k) \geq 0,$ which implies $v(z) \geq 0.$
\end{proof}

\begin{thm} If $P(z)$ is a polynomial with coefficients in $\mathbb{Z}$ and $a\in \overline{\mathbb{Q}}$ is an algebraic integer, then the $2$-adic valuation of $P(a)$ is nonnegative. 
\end{thm}

\begin{proof} Since $v(z) \geq 0$ and $w= a_{n} z^n + \cdots + a_1 z + a_0, $ it follows that  $$v(w) = \min_{0 \leq j \leq n} v(a_j) + jv(z) \geq 0.$$\end{proof}

\subsection{The main goal.}
We will assume throughout this section that $a$ is a Misiurewicz parameter for the semigroup family $\langle z^2 +c, -z^{2} +c \rangle.$

Therefore, the critical point zero has a unique bounded orbit, which is necessarily pre-periodic:
\[
\check{z}_0 = 0 \xrightarrow{\mathbf{f}_a} \check{z}_1 = a \xrightarrow{\mathbf{f}_c} \check{z}_2 \xrightarrow{\mathbf{f}_c} \cdots
\]
such that the critical point is eventually mapped to the  cycle $\alpha_a$ associated with $a$, which we denote by \[\alpha_a=(\check{z}_j)_{\check{\ell}}^{\check{\ell} + n}. \]  The correspondence with the previous notation in \eqref{alkjsokjdkkdf}  is given by $z_j = \check{z}_{j+1}$ and $\ell = \check{\ell} - 1$.
 The   orbit $(\check{z}_j)_0^\infty$ is completely determined by a sequence of signs $\sigma_j \in \{-1, 1\}$ satisfying
\begin{equation} \label{gdsgfcxewer}
\check{z}_{j+1} = \sigma_j \check{z}_j^2 + c, \quad \check{z}_0 = 0.
\end{equation}


By Remark~\ref{gjdjwew}, the smallest positive integer $\check{\ell}$ such that $\check{z}_{\check{\ell} + n} = \check{z}_{\check{\ell}}$ satisfies
\[
\check{\ell} = \ell + 1 \geq 2.
\]
We will refer to $\check{\ell}$ and $n$, with this precise meaning, throughout this section.

Using the sequence $\sigma_j$ determined by \eqref{gdsgfcxewer} we inductively define a sequence of polynomials $F_j(c)$ by setting $F_1(c)=c$ 
 and \begin{equation} \label{gjdse} F_{j+1}(c) = \sigma_j F_j(c)^2 +c. \end{equation}

 The sequence $F_j$ can used to give a simplified form of the function $w(c)$ in Definition \ref{asdljhafspdoifasdf}. We have $w(c) =F_{n+\check{\ell}}(c) - F_{\check{\ell}}(c), $ for every $c$ in a neighborhood of $a.$ In this way: {\it  the main goal  is to show that the derivative of $c \mapsto F_{n+\check{\ell}}(c) -F_{\check{\ell}}(c)$ at $c=a$ is nonzero. }

\subsection{Preliminary lemmas.} Let $\overline{\mathbb{Q}}$ denote the algebraic closure of $\mathbb{Q}.$ Let $v=v_2$ denote the $2$-adic valuation.
Let $\mathfrak{p}=\{z \in \overline{\mathbb{Q}} \mid v(z)>0\},$  $\mathfrak{n}=\{z \in \overline{\mathbb{Q}} \mid v(z) \geq 0\}$ and $2 \mathfrak{n} = \{ 2z \mid z\in \mathfrak{n}\}.$ Notice that $2\mathfrak{n}$ is a subset of $\mathfrak{p}.$

\begin{lem} \label{gkdkgkew} Assume $\check{\ell} \geq 3$. Suppose that  $\sigma_{\check{\ell} +n -1}= \sigma_{\check{\ell} -1}.$ Then 

\begin{equation} \label{ggkkk} F_{n+\check{\ell}-2}(a) \equiv F_{\check{\ell}-2}(a) \bmod \mathfrak{p},\end{equation}

\begin{equation}  \label{gjmcjd} F_{n+ \check{\ell} -2}'(a) \equiv 1 \bmod 2\mathfrak{n},
\end{equation}

\begin{equation}  \label{gjjd}   F_{\check{\ell}-2}'(a) \equiv 1 \bmod 2\mathfrak{n}.
\end{equation}

\end{lem}

\begin{proof} Since $F_{\check{\ell} +n} (a) = F_{\check{\ell}}(a)$ and  $\sigma_{\check{\ell} +n -1}= \sigma_{\check{\ell} -1}, $
 we have

$$\sigma_{\check{\ell} +n -1}F_{\check{\ell} + n -1}(a)^{2} +a = \sigma_{\check{\ell} -1}F_{\check{\ell} -1}(a)^2 +a.
$$ Hence 
$F_{\check{\ell} + n -1}(a)^{2} = F_{\check{\ell} -1}(a)^2$ and 
\[ (F_{\check{\ell} +n -1}(a) - F_{\check{\ell} -1}(a))(F_{\check{\ell} +n -1}(a) + F_{\check{\ell} -1}(a)) =0.\]

\noindent Since $\check{\ell}$ is minimal, the first term in the product is nonzero. Thus 
\begin{equation} G_{0}(c):= F_{\check{\ell} + n -1}(c) + F_{\check{\ell} -1}(c)
\end{equation} vanishes at $c=a.$  Using the same idea, we can express $F_{\check{\ell}+n-1}(a)$ and $F_{\check{\ell}-2}(a)$  in terms of the square of another function, and therefore

\begin{equation} \label{(3)}\sigma_{\check{\ell}+n-2} F_{\check{\ell}+n-2}(a)^2+\sigma_{\check{\ell}-2} F_{\check{\ell}-2}(a)^2=-2 a.
\end{equation}

Let $\zeta=F_{\check{\ell}+n-2}(a),$ $\eta=F_{\check{\ell} -2}(a),$ $ \sigma=\sigma_{\check{\ell}+n-2}$ and $ \tau=\sigma_{\check{\ell}-2}.$ Since $a$ is an algebraic integer and $F_j(c)$ is a polynomial for every $j,$ it follows that the $2$-adic valuations of $\zeta$ and $\eta$ are nonnegative. Hence $\zeta, \eta \in \mathfrak{n}.$ 
Each of the signs $\sigma$ and $\tau$ is either $-1$ or $1.$ We have four possibilities. 
\begin{enumerate} \item[$(a)$] If $(\sigma, \tau)=(1,1)$ then $\zeta^2+\eta^2=-2 a$. We have

\end{enumerate}
$$
\begin{aligned}
& (\zeta-\eta)^2=\zeta^2+\eta^2-2 \eta \zeta=-2 a-2 \eta \zeta \\
& 2 v(\zeta-\eta) \geq 1 \Rightarrow v(\zeta-\eta)>0 .
\end{aligned}
$$
(In the preceding calculation, we have used the properties described previously: $v(z^n)=nv(z),$ and $v(z +w)$  is at least $\min \{v(z), v(w) \}.$)

\begin{enumerate} \item[$(b)$] If $(\sigma, \tau)=(1,-1)$ then from \eqref{(3)} we have $\zeta^2-\eta^2=-2 a.$  Then 
\end{enumerate}
 $$(\zeta-\eta)(\zeta+\eta)=-2 a.$$ After applying the $2$-adic valuation on both sides we get  $v(\zeta-\eta)+v(\zeta+\eta)>0$.
Since $v(-\eta)=v(\eta)$,
$v(\zeta-\eta)$ equals $\min \{v(\zeta), v(-\eta)\}$, which is the same as $\min \{v(\zeta), v(\eta)\}=v(\zeta+\eta).$

Hence $v(\zeta - \eta) >0$ in the second case as well.

\begin{enumerate} \item[$(c)$] If $(\sigma, \tau)=(-1,-1),$ by \eqref{(3)} we have $-\zeta^2-\eta^2=-2 a$. From $(\zeta-\eta)^2=2 a-2 \zeta \eta$ we conclude that $v(\zeta-\eta)>0$.
\end{enumerate}

\begin{enumerate}\item[$(d)$] If $\left(\sigma, \tau\right)=(-1,1),$ then $-\zeta^2+\eta^2=-2 a$ and
$ (\eta-\zeta)(\eta+\zeta)=-2 a.$ Hence $2 v(\eta-\zeta)>0,$ which implies $v(\eta-\zeta)>0 .
$

\end{enumerate}

\noindent We  conclude that $\zeta \equiv \eta \bmod \mathfrak{p}$ in all four possibilities. The second equation \eqref{gjmcjd} can be verified by expressing $F_{n + \check{\ell} -2}(c)$ in terms of the square of  $F_{n+ \check{\ell} -3}(c).$ Hence

$$F_{n + \check{\ell} -2}'(a) = 2 \sigma_{n+\check{\ell} -3}F_{n+\check{\check{\ell}} -3}(c) F_{n+\check{\ell} -3}'(a) +1. $$

\noindent The third equation \eqref{gjjd} can be checked in the same way if $\check{\ell}\geq 4$. If $\check{\ell}=3$ the proof is simpler, for then $F_{\check{\ell} -2}'(a)=1$. 
\end{proof}

\begin{lem} \label{hgeqdbc}   Define 
$G_0(c) = F_{\check{\ell} +n -1}(c) + F_{\check{\ell} -1}(c),$
 \noindent for every $c$ in a neighbourhood of $a,$ where $\check{\ell} \geq 2.$  Suppose that  $\sigma_{\check{\ell} +n -1}= \sigma_{\check{\ell} -1}.$ Then
 
  \begin{equation} \frac{G_0'(a)}{2} \equiv 1 \bmod \mathfrak{p}.\end{equation}

\end{lem}

\begin{proof}  We will separate the proof into two cases according to whether $\check{\ell} =2$ or $\check{\ell} \geq 3.$ In the first case,  $\check{\ell}=2.$ Then 

$$G_0(c) = F_{n+1}(c) + c = \sigma_n F_{n}(c)^2 + 2c.$$  

\noindent Since $G_0(a)=0$, we have $\sigma_nF_n(a)^2 + 2a =0$, and therefore $v(F_n(a)) >0.$ It follows that the valuation of  $\sigma_n F_n(a) F_n'(a)$ is positive (the valuation of the product is the sum of the valuations.) We conclude that 

$$\frac{G_0'(a)}{2} = \sigma_n F_{n}(c) F_{n}'(c) +1 \equiv 1 \bmod \mathfrak{p}. $$

\noindent In the second case,  $\check{\ell} \geq 3.$ In this case, it is possible to write $F_{\check{\ell} -1}(c)$ in terms of the square of $F_{\check{\ell} -2}(c)$. Therefore
$ G_0(c) = \sigma_{\check{\ell} +n -2}F_{\check{\ell} + n -2}(c) ^2 + \sigma_{\check{\ell} -2}(c)^2 + 2c$ and 
\begin{equation} \label{(5)}
\frac{G_0'(a)}{2} = \sigma_{\check{\ell} + n-2}F_{\check{\ell} + n -2}(a) F_{\check{\ell} + n -2}'(a) + \sigma_{\check{\ell} -2} F_{\check{\ell} -2}(a)F_{\check{\ell} -2}'(a)  +1. 
\end{equation}

\noindent Let $\zeta = F_{\check{\ell} + k -2}(a).$ By Lemma \ref{gkdkgkew}, the value of $F_{\check{\ell} -2}(a)$ is $\zeta + m_1$ where $v(m_1) >0.$ Moreover, 
$ F_{\check{\ell} -2}'(a) = 1 + 2b_0$ and  $ F_{\check{\ell} + n -2}'(a) = 1 + 2b_1$
where $v(b_0) \geq 0$ and $v(b_1) \geq 0.$  By \eqref{(5)}, ${G_0'(a)}/2$ equals

$$ \sigma_{\check{\ell} + n -2}\zeta (1+ 2b_1) + \sigma_{\check{\ell} -2} (1+ 2b_0)(\zeta + m_1) +1, $$

\noindent  which can be simplified into
 \[\sigma_{\check{\ell} + n -2} \zeta + 2\sigma_{\check{\ell} + n -2}\zeta b_1 + \sigma_{\check{\ell} -2}\zeta + \sigma_{\check{\ell}-2}m_1 + 2\sigma_{\check{\ell} -2} b_0 \zeta + 2\sigma_{\check{\ell} -2}b_0m_1 +1. \] Consequently,  ${G_0'(a)}/{2}$ equals
\[ 2(\sigma_{\check{\ell} + n-2}\zeta b_1 + \sigma_{\check{\ell} -2} b_0 \zeta + \sigma_{\check{\ell} -2}b_0 m_1) + \zeta (\sigma_{\check{\ell} + k -2} + \sigma_{\check{\ell} -2}) + \sigma_{\check{\ell} -2}m_1 +1.\]
The last sum consists of four terms. The first clearly has positive valuation, the second may be $0$, $2\zeta$ or $-2\zeta$, and therefore has positive valuation. The third has positive valuation, since $v(m_1) >0.$ The fourth term is $1.$ Since the valuation of the sum is at least the minimum of the valuations, it follows that 
$ {G_0'(a)}/{2} \equiv 1 \bmod \mathfrak{p}.$\end{proof}

\begin{lem} \label{ghobcsgasd} Suppose that $\sigma_{\check{\ell} +n -1}= - \sigma_{\check{\ell} -1}$ and $\check{\ell} \geq 2.$   Define \[G_1(c) = F_{\check{\ell} +n -1}(c) + iF_{\check{\ell} -1}(c),\]   \[G_2(c) = F_{\check{\ell} +n -1}(c) -i F_{\check{\ell} -1}(c),\] for every $c$ in a neighbourhood of $a.$   If $G_1(a)=0,$ then 

\begin{equation}G_1'(a) \equiv (1+i) \bmod \mathfrak{p}.
\end{equation} 

\noindent If $G_2(a) =0,$ then 
\begin{equation} \label{gdiigd} G_2'(a) \equiv (1-i) \bmod \mathfrak{p}. 
\end{equation}

\end{lem}

\begin{proof}  Suppose that $G_1(a)=0.$ If $\check{\ell}=2,$ then $G_1(c) = \sigma_kF_{k}(c)^2 + c + ic, $ and     $G_1'(a) = 2\sigma_k F_k(a) F_k'(a) + (1+i).$ If $\check{\ell} \geq 3,$  then $G_1'(a)$ is given by 
$$2(\sigma_{\check{\ell} +n -2} F_{\check{\ell} + n -2}(a) F_{\check{\ell} + n -2}'(a) + i \sigma_{\check{\ell} -2} F_{\check{\ell} -2}(a) F_{\check{\ell} -2}'(a)) + 1 +i. $$\noindent In any case, we have $G_1'(a) \equiv (1+i) \bmod \mathfrak{p}$ whenever $G_1(a) =0.$

The same argument can be used to prove \eqref{gdiigd}.
\end{proof}

\begin{proof}[\bf Proof of Theorem \ref{dfasdfefd}]   Since $a$ is a Misiurewicz point,
$ F_{\check{\ell} +a}(a)=F_{\check{\ell}}(a).$  
Since $F_{j+1}(a)=\sigma_jF_{j}(a)^2 +a,$ where $\sigma_j$ belongs to $ \{-1, 1\},$ it follows that
\[\sigma_{\check{\ell} +n -1}F_{\check{\ell} + n -1}(a)^{2} +a = \sigma_{\check{\ell} -1}F_{\check{\ell} -1}(a)^2 +a, \] 
\begin{equation}\label{ghdkkkse}
\sigma_{\check{\ell} +n -1}F_{\check{\ell} + n -1}(a)^{2} = \sigma_{\check{\ell} -1}F_{\check{\ell} -1}(a)^2. 
\end{equation}

We have two cases depending on the signs of $\sigma_{\check{\ell}+ n -1}$  and $\sigma_{\check{\ell}-1}.$ If they coincide, then  $F_{\check{\ell} +n -1}(a)^2 = F_{\check{\ell} -1}(a)^2$ and  
 $$  (F_{\check{\ell} +n -1}(a) - F_{\check{\ell} -1}(a))(F_{\check{\ell} +n -1}(a) + F_{\check{\ell} -1}(a)) =0.$$ Since $\check{\ell}$ is minimal, the second factor must vanish at $a.$  By Lemma  \ref{hgeqdbc}, the derivative of $G_0(c) = F_{\check{\ell} +n -1}(c) + F_{\check{\ell} -1}(c)$ satisfies $G_0'(a)/2 \equiv 1 \bmod \mathfrak{p}.$ Recall that the $2$-adic valuation of every element of $\mathfrak{p}$ is strictly positive. Therefore, if $\sigma_{\check{\ell} + n -1}$ coincides with $\sigma_{\check{\ell} -1},$ then $G_0'(a) \neq 0$ and 
 \begin{equation}
 \begin{split}
 F_{\check{\ell} +n}(c) - F_{\check{\ell}}(c) &= \sigma_{\check{\ell} + n -1} F_{\check{\ell} +n -1}(c)^2 + c - \sigma_{\check{\ell} -1}F_{\check{\ell} -1}(c)^2 -c \\
 & = \sigma_{\check{\ell} -1}(F_{\check{\ell} + n -1}(c)^2 - F_{\check{\ell} -1}(c)^2) \\
 & = \sigma_{\check{\ell} -1}(F_{\check{\ell} + n -1}(c) + F_{\check{\ell} -1}(c)) (F_{\check{\ell} + n -1}(c) - F_{\check{\ell} -1}(c))\\
 & = \sigma_{\check{\ell} -1}G_{0}(c) (F_{\check{\ell} + n -1}(c) - F_{\check{\ell} -1}(c)). 
 \end{split}
 \end{equation} It follows that 
 $ F_{\check{\ell} + n}'(a) - F_{\check{\ell}}'(a) = -2\sigma_{\check{\ell} -1} G_{0}'(a) F_{\check{\ell} -1}(a)$ which is  $\neq 0.$ Here  $F_{\check{\ell} -1}(a) \neq 0$ because the bounded orbit of $0$ is strictly pre-periodic. 
 
There is nothing else to prove in the case $\sigma_{\check{\ell} + n -1} = \sigma_{\check{\ell} -1}.$ 
If $\sigma_{\check{\ell} + n -1} = -\sigma_{\check{\ell} -1},$ then by \eqref{ghdkkkse} we have
 $(F_{\check{\ell} +n -1}(a)^2 + F_{\check{\ell}-1}(a)^2)=0.$
 
\noindent Either  $G_1(a)=F_{\check{\ell} +n -1}(a) + i F_{\check{\ell}-1}(a)$ or  $G_{2}(a) = F_{\check{\ell} +n-1}(a) - i F_{\check{\ell} -1}(a)$ is zero. By Lemma \ref{ghobcsgasd}, if $G_1(a)=0$ then

$$F_{\check{\ell} +n -1}'(a) - F_{\check{\ell}}'(a) = -2\sigma_{\check{\ell} + n -1} i G_1'(a) F_{\check{\ell} -1}(a) \neq 0. $$

  \noindent If $G_2(a)=0$, a similar argument shows that $F_{\check{\ell} +n}'(a) - F_{\check{\ell}}'(a) \neq 0.$
\end{proof}

\subsection*{Acknowledgments}The author would especially like to thank Luna Lomonaco for inspiring discussions on this subject in late 2024.
This research was partially supported by CNPq/MCTI/FNDCT, project 406750/2021-1 (Brazil).
The author also thanks Daniel Smania for his kind hospitality at the Institute of Mathematics and Computer Sciences, University of S\~ao Paulo.

\bibliography{oi}




\end{document}